\documentclass[12pt]{article}

\usepackage[a4paper,margin=1in]{geometry}
\usepackage{amsmath,amsthm,amssymb,mathtools}
\usepackage{array,booktabs,longtable,multirow}
\usepackage{graphicx}
\usepackage{xcolor}
\usepackage{tikz,cite,pgfplots}
\pgfplotsset{compat=1.18}
\usetikzlibrary{arrows.meta,positioning,calc,shapes.geometric}
\usepackage{hyperref}
\usepackage{algorithm}
\usepackage{algpseudocode}
\usepackage{enumitem}
\usepackage{caption,booktabs,tabularx}
\everymath{\displaystyle}
\renewenvironment{proof}{{\noindent \bf Proof.}}{\qed}

\hypersetup{
	colorlinks=true,
	linkcolor=blue!60!black,
	citecolor=blue!60!black,
	urlcolor=blue!60!black
}

\newtheorem{theorem}{Theorem}[section]
\newtheorem{proposition}[theorem]{Proposition}
\newtheorem{lemma}[theorem]{Lemma}
\newtheorem{corollary}[theorem]{Corollary}

\theoremstyle{definition}
\newtheorem{definition}[theorem]{Definition}
\newtheorem{remark}[theorem]{Remark}
\newtheorem{example}[theorem]{Example}
\newtheorem{openproblem}[theorem]{Open Problem}

\DeclareMathOperator{\reg}{reg}
\DeclareMathOperator{\pd}{pd}
\DeclareMathOperator{\depth}{depth}
\DeclareMathOperator{\Hilb}{Hilb}
\DeclareMathOperator{\Ind}{Ind}

\newcommand{\join}{\nabla}

\title{Homological invariants of edge ideals of the multiple extended complete split-like graphs}

\author{Bilal Ahmad Rather\\[2mm]
	\small School of Mathematics and Statistics, Shandong University of Technology,\\
	\small Zibo 255049, China\\
	\texttt{\href{mailto:bilalahmadrr@gmail.com}{bilalahmadrr@gmail.com}}
			}

\date{}
\begin{document}
	\maketitle	
	\begin{abstract}
		We study the graphs $MECS_{b,n}^a \cong \overline{K}_a \join \big(n(K_b+K_2)\big)$, obtained by attaching an independent set of size $a$ to $n$ disjoint copies of the block $K_b+K_2$. For $n=1$, we get $MECS_{b,1}^a$, and  recover the results of one-block case studied in [Anand, Gupta, Rather and Singh, Homological invariants of some complete split-like graphs, Beitr. Algebra Geom. (2025)]. Using Hochster's formula, tensor products of minimal free resolutions over disjoint variable sets, and the Betti-number formula for graph joins, we derive explicit descriptions of the independence complex, independence polynomial and its analytic properties, Hilbert series, linear and quadratic Betti strands, regularity, projective dimension, and several structural invariants of $MECS_{b,n}^a$. We further classify the well-covered and unmixed members, compute induced matching numbers, show that the family is never Cohen--Macaulay, and record algorithmic procedures for evaluating the Betti data.
	\end{abstract}
	
	
	
	\section{Introduction}\label{sec:introduction}
	
	The passage from a finite simple graph to its edge ideal has become one of the most productive interfaces between combinatorics and commutative algebra. If $G$ is a finite simple graph with vertex set $V(G)=\{x_1,\dots,x_r\}$, then the square-free quadratic monomial ideal $ I(G)=\big(x_ix_j:{x_i,x_j}\in E(G)\big) $ packages the adjacency data of $G$ into a form accessible to homological methods. The minimal graded free resolution of $R/I(G)$, where $R=K[x_1,\dots,x_r]$, encodes graded Betti numbers, projective dimension, regularity, and several finer measures of complexity. This circle of ideas goes back to the foundational work of Hochster, Stanley, Villarreal, and many others, and remains central in current studies of monomial ideals, simplicial complexes, and graph parameters \cite{Hochster1975, Stanley1996, MillerSturmfels2005, HerzogHibi2011, Villarreal2015}. The attraction of the subject is that algebraic invariants can often be computed from graph-theoretic information, yet the resulting formulas are rarely obvious from either side alone \cite{FranciscoHaVanTuyl2009,Froberg1990,HaVanTuyl2008,HerzogHibiZheng2004,Jacques2004,Jonsson2008,Katzman2006,HaWoodroofe2014}.
	
	Among graph classes that are particularly suitable for this interplay, split graphs and their extensions occupy a useful middle ground. On the one hand, they are structured enough to admit exact enumeration arguments. On the other hand, they are rich enough to exhibit nontrivial homological phenomena. Complete split graphs, nearly complete split graphs, and related families have already been examined from the viewpoint of graded Betti numbers, regularity, and projective dimension by several authors \cite{AnandSinghVats2025, AnandGuptaRatherSingh2026,GuptaSinghAnand2026}. Recent work also points to broader connections with domination parameters, shellability, and Cohen--Macaulay type behavior \cite{DaoSchweig2013, MoreyVillarreal2012, Woodroofe2009,bilafilo,bilalfilo1,bilaljco}. What emerges from these studies is a recurring pattern: once a graph family is built from a few elementary pieces by join, disjoint union, or Cartesian-type operations, the corresponding algebra tends to inherit highly organized syzygy patterns.
	
	 Suppose $G_1=(V(G_1),E(G_1))$ and $G_2=(V(G_2),E(G_2))$ are simple graphs with disjoint vertex sets, that is,
	$ V(G_1)\cap V(G_2)=\emptyset. $ We use the following two standard constructions. The \textit{join} of $G_1$ and $G_2$, written as $G_1 * G_2$, is the graph whose vertex set is the disjoint union
		$ V(G_1)\sqcup V(G_2), $ and whose edge set is obtained by taking all edges already present in $G_1$ and $G_2$, together with every edge joining a vertex of $G_1$ to a vertex of $G_2$,  that is, 
		$$E(G_1 * G_2)
		=
		E(G_1)\cup E(G_2)\cup
		\bigl\{\{x,y\}\mid x\in V(G_1),\ y\in V(G_2)\bigr\}.
		$$
		 The \textit{sum} of $G_1$ and $G_2$, denoted by $G_1+G_2$, is the graph with vertex set
		$ V(G_1)\times V(G_2). $ Two vertices $(u_1,u_2)$ and $(v_1,v_2)$ are adjacent precisely when one coordinate is fixed and the other follows an edge from the corresponding graph; that is, adjacency occurs exactly when
		$ u_1=v_1 \ \text{and}\ (u_2,v_2)\in E(G_2), $
		or $
		u_2=v_2 \ \text{and}\ (u_1,v_1)\in E(G_1).
		$ 
	The present paper is devoted to a family that naturally combines two directions already visible in the literature. Let
	$ B_b:=K_b+K_2, $ where $+$ denotes the graph-theoretic as above, see \cite{AnandGuptaRatherSingh2026}. Concretely, $B_b$ may be viewed as two copies of $K_b$ joined by a perfect matching. If $nB_b$ denotes the disjoint union of $n$ copies of $B_b$, then the \emph{multiple extended complete split-like graph} is
	$ MECS_{b,n}^a\cong \overline{K}_a \join (nB_b). $
	Thus one starts with an independent set of size $a$, takes $n$ disconnected copies of the block $B_b$, and joins every vertex in the independent set to every vertex in every block. This construction simultaneously generalizes two known families. When $n=1$, one recovers the extended complete split-like graph studied in \cite{AnandGuptaRatherSingh2026}. When $B_b$ is replaced by $K_b$, one obtains the multiple complete split-like graph treated in the same source and, in special cases, complete split graphs and windmill graphs \cite{AnandGuptaRatherSingh2026}. The repeated $B_b$-block family therefore sits exactly at the intersection of two well-motivated directions, but its explicit homological behavior has not been developed.
	
	There is also a technical reason to expect new behavior. A single block $B_b$ already has a two-strand Betti table in the quotient resolution \cite{AnandGuptaRatherSingh2026}. Repeating the block by disjoint union does not merely enlarge the graph. It propagates the two-strand structure across tensor products, so the total regularity grows by a different law from the clique-block family $\overline{K}_a\join nK_b$. In other words, replacing $K_b$ by $K_b+K_2$ is not a cosmetic modification. It changes the independence complex from a simplex to a crown-type complex, and that change forces a different pattern in the Hilbert series, in the strand support of the minimal resolution, and in covering properties such as well-coveredness and unmixedness. This is precisely the type of phenomenon that makes split-like families useful test objects in combinatorial commutative algebra. From a more applied viewpoint, the graph $MECS_{b,n}^a$ can be read as a modular architecture. The independent vertices behave like an external supervisory layer, while each block $B_b$ is a replicated internal module made of two dense local clusters linked by a matching. Repeating the block $n$ times preserves the local geometry inside each module but changes the global syzygy profile through tensor-product effects. Such replicated constructions are algebraically interesting because they isolate how invariants behave under controlled scaling. In particular, they make it possible to separate the effect of the join operation from the effect of block replication.
	
	The main purpose of this paper is to give a systematic and self-contained treatment of this family. Our starting point is the explicit Betti data of a single block $B_b$ obtained in \cite{AnandGuptaRatherSingh2026}. From there we derive a Betti polynomial for $nB_b$ by means of tensor products of minimal free resolutions over disjoint variable sets. We then feed this information into the join formula of Mousivand \cite{Mousivand2012} to obtain explicit Betti formulas for $MECS_{b,n}^a$. In addition, we determine structural invariants such as the order, size, degree sequence, clique number, chromatic number, independence number, and vertex cover number. We also compute the independence polynomial and Hilbert series and prove exact formulas for projective dimension, regularity, induced matching number, and the well-covered criterion. In particular, for $b\geq 3$ we show that the regularity of $MECS_{b,n}^a$ equals $2n$, while the corresponding clique-block family has regularity $n$ \cite{AnandGuptaRatherSingh2026}. We also show that the well-covered threshold shifts from $a=n$ in the clique-block setting to $a=2n$ in the present family, exactly reflecting the fact that each block $B_b$ contributes an independent set of size $2$ rather than $1$.
	
	For the reader's convenience, we isolate the principal contributions as: We describe the independence complex of $MECS_{b,n}^a$ as the disjoint union of a simplex and an $n$-fold simplicial join of crown complexes. We derive explicit formulas for the independence polynomial along with its analytic properties and Hilbert series and obtain closed expressions for the numbers of independent sets of each size.  We compute the Betti polynomial of $nB_b$ and then obtain formulas for the linear and quadratic Betti strands of $MECS_{b,n}^a$.  We determine $\reg(MECS_{b,n}^a)$, $\pd(MECS_{b,n}^a)$, the induced matching number, the well-covered criterion, and several non-Cohen--Macaulay consequences. We provide an algorithmic scheme for generating the Betti data and illustrate the theory by worked examples and comparison tables.

	Organization of paper: In Section~\ref{sec:preliminaries}, we records the notation, recalls the needed facts on edge ideals, graph joins, and simplicial complexes, and fixes the single-block Betti data used later. Section~\ref{sec:independence} develops the basic structure theory of $MECS_{b,n}^a$, describes the independence complex, counts independent sets, derives the independence polynomial and Hilbert series, and obtains the top Betti number, projective dimension, and depth. Section~\ref{sec:linear} introduces a Betti polynomial for repeated blocks, proves a product formula for disjoint unions of the block $B_b$, and deduces the linear strand of $MECS_{b,n}^a$. Section~\ref{sec:higher} studies higher strands, with particular emphasis on the quadratic strand, the support window of the Betti table, and the regularity of the family. Section~\ref{sec:matching} turns to covering and matching phenomena, induced matching number, purity, well-coveredness, unmixedness, failure of Cohen--Macaulayness, and related consequences are determined there. Section~\ref{sec:algorithm} gives an algorithmic implementation of the formulas, presents a computational flow diagram, and tabulates representative examples together with comparisons against previously studied split-like families. Section \ref{ind section} gives the analytic properties of $MECS_{b,n}^a$.  Finally, Section~\ref{sec:conclusion} summarizes the contributions, discusses the limitations of the present treatment, and outlines several directions for further work.
	
	\section{Preliminaries and notation}\label{sec:preliminaries}
	
	Throughout the paper, all graphs are finite, simple, and undirected. We work over a field $K$. For a graph $G$ on the vertex set $V(G)={x_1,\dots,x_r}$, we write
	$$
	R_G:=K[x_1,\dots,x_r], \qquad I(G):=\big(x_ix_j:{x_i,x_j}\in E(G)\big)\subseteq R_G.
	$$
	The quotient $R_G/I(G)$ is the edge ring of $G$. Its graded Betti numbers are denoted by $\beta_{i,j}(G)$, and we write
	$$
	\pd(G):=\pd(R_G/I(G)), \qquad \reg(G):=\reg(R_G/I(G)).
	$$
	We adopt the usual convention that $\binom{u}{v}=0$, whenever $v<0$ or $v>u$.
	
	The next definition fixes the graph operations used throughout the paper.
	\begin{definition}
		 Let $G$ and $H$ be graphs on disjoint vertex sets.
		\begin{enumerate}[label=\textup{(\roman*)},leftmargin=8mm]
			\item The \emph{disjoint union} $G\sqcup H$ is the graph with vertex set $V(G)\sqcup V(H)$ and edge set $E(G)\sqcup E(H)$.
			\item The \emph{join} of $G$ and $H$, denoted by $G\join H$, is obtained from $G\sqcup H$ by adding every edge between $V(G)$ and $V(H)$.
			\item For a graph $H$, the notation $nH$ means the disjoint union of $n$ copies of $H$.
		\end{enumerate}
	\end{definition}
	
	The next definition identifies the basic block used later.
	\begin{definition}
		For $b\geq 2$, let $B_b:=K_b+K_2$ be the graph with vertex set
		$ V(B_b)=\{x_1,\dots,x_b,y_1,\dots,y_b\} $
		and edge set
		$$
		E(B_b)=\big\{\{x_i,x_j\}:1\leq i<j\leq b\big\}\cup
		\big\{\{y_i,y_j\}:1\leq i<j\leq b\big\}\cup
		\big\{\{x_i,y_i\}:1\leq i\leq b\big\}.
		$$
		Thus $B_b$ consists of two copies of $K_b$ joined by a perfect matching, see Figure \ref{fig0}.
	\end{definition}
	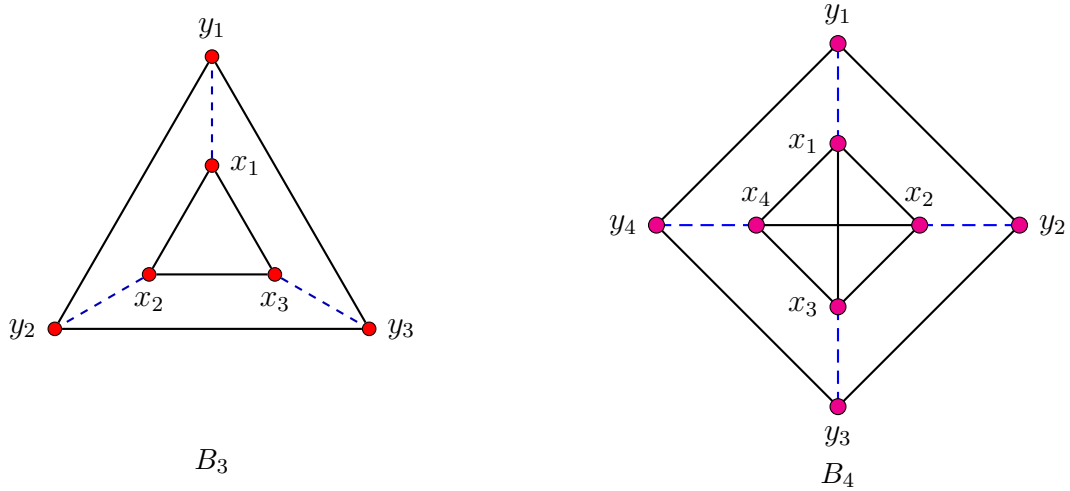
\begin{figure}[h]
		\centering
		\begin{minipage}{0.48\textwidth}
			\centering
			\begin{tikzpicture}[
				scale=1.2, 
				v/.style={circle, draw, fill=red, inner sep=1.8pt},
				matching/.style={thick, dashed, blue!70!black}
				]
				\node[v, label=right:$x_1$] (x1) at (90:0.8) {};
				\node[v, label=below:$x_2$] (x2) at (210:0.8) {};
				\node[v, label=below:$x_3$] (x3) at (330:0.8) {};
				
				\node[v, label=above:$y_1$] (y1) at (90:2.0) {};
				\node[v, label=left:$y_2$] (y2) at (210:2.0) {};
				\node[v, label=right:$y_3$] (y3) at (330:2.0) {};
				
				\draw[thick] (x1) -- (x2) -- (x3) -- (x1);
				\draw[thick] (y1) -- (y2) -- (y3) -- (y1);
				\draw[matching] (x1) -- (y1);
				\draw[matching] (x2) -- (y2);
				\draw[matching] (x3) -- (y3);
				
				\node[anchor=north, font=\small] at (0,-2.2) {$B_3$};
			\end{tikzpicture}
		\end{minipage}
		\hfill 
		\begin{minipage}{0.48\textwidth}
			\centering
			\begin{tikzpicture}[
				scale=1.2, 
				vertex/.style={circle, draw, fill=magenta, inner sep=0pt, minimum size=6pt},
				edge/.style={thick, black},
				matching/.style={thick, blue, dash pattern=on 5pt off 3pt}
				]
				\def\b{4}
				\def\innerRadius{0.9}
				\def\outerRadius{2.0}
				
				\foreach \i in {1,...,\b} {
					\coordinate (x\i) at ({90 + (\i-1)*-90}:\innerRadius);
					\coordinate (y\i) at ({90 + (\i-1)*-90}:\outerRadius);
				}
				
				\foreach \i in {1,...,\b} { \draw[matching] (x\i) -- (y\i); }
				
				\foreach \i in {1,...,\b} {
					\foreach \j in {\i,...,\b} {
						\ifnum \i<\j \draw[edge] (x\i) -- (x\j); \fi
					}
				}
				
				\draw[edge] (y1) -- (y2) -- (y3) -- (y4) -- (y1);
				
				\node[vertex, label={left:$x_1$}] at (x1) {};
				\node[vertex, label={above:$x_2$}] at (x2) {};
				\node[vertex, label={left:$x_3$}] at (x3) {};
				\node[vertex, label={above:$x_4$}] at (x4) {};
				
				\node[vertex, label={above:$y_1$}] at (y1) {};
				\node[vertex, label={right:$y_2$}] at (y2) {};
				\node[vertex, label={below:$y_3$}] at (y3) {};
				\node[vertex, label={left:$y_4$}] at (y4) {};
				
				\node[anchor=north, font=\small] at (0,-2.5) {$B_4$};
			\end{tikzpicture}
		\end{minipage}
		\caption{Graphs $B_{3}$ and $B_{4}$.}
		\label{fig0}
	\end{figure}
	The next definition introduces the graph family studied in this article.
	\begin{definition}
		For integers $a\geq 1$, $b\geq 2$, and $n\geq 1$, we define
		$ MECS_{b,n}^a:=\overline{K}_a \join (nB_b). $
		Whenever convenient, we abbreviate this graph by $G_{a,b,n}$, see Figure \ref{fig:mecsgraph}.
	\end{definition}
	
	The next definition recalls the simplicial objects attached to a graph.
	\begin{definition}	
		The independence complex of a graph $G$ is
		$$
		\Delta(G):=\{F\subseteq V(G): F \text{ is an independent set of }G\}.
		$$
		If $\Delta_1$ and $\Delta_2$ are simplicial complexes on disjoint vertex sets, their simplicial join is
		$$
		\Delta_1\star \Delta_2:=\{F_1\cup F_2:F_1\in \Delta_1,\ F_2\in \Delta_2\}.
		$$
		For a graph $G$, we also write independence polynomial as
		$$
		\Ind(G,z):=\sum_{k\geq 0} s_k(G)z^k,
		$$
		where $s_k(G)$ denotes the number of independent sets of size $k$ in $G$.
	\end{definition}
	
	The following theorem is the standard combinatorial formula for graded Betti numbers of Stanley--Reisner rings, which turns reduced homology of induced subcomplexes into Betti numbers.	
	\begin{theorem}[Hochster's formula, see \cite{Hochster1975}]		
		Let $\Delta$ be a simplicial complex on the vertex set $[r]={1,\dots,r}$. Then
		$$
		\beta_{i,d}(K[\Delta])=
		\sum_{\substack{W\subseteq [r]\\ |W|=d}}
		\dim_K \widetilde{H}_{d-i-1}(\Delta[W];K).
		$$
	\end{theorem}
	In practice, for a graph $G$, we may compute $\beta_{i,d}(G)$ from the independence complex $\Delta(G)$.

	The next theorem is the join formula that we shall use after the repeated-block convolution has been computed.
	
	\begin{theorem}[\cite{Mousivand2012}]\label{Mousivand2012 theorem 2.2}
		Let $G$ and $H$ be simple graphs with disjoint vertex sets of cardinalities $m$ and $n$, respectively. Then
		$$
		\beta_{i,d}(G\join H)=
		\sum_{j=0}^{d-2}
		\left(
		\binom{n}{j}\beta_{i-j,d-j}(G)+
		\binom{m}{j}\beta_{i-j,d-j}(H)
		\right)
		+\delta_{d,i+1}\sum_{j=1}^{d-1}\binom{m}{j}\binom{n}{d-j},
		$$
		where $\delta_{p,q}$ denotes the Kronecker delta and all out-of-range Betti numbers are understood to be zero.
	\end{theorem}
	
	The next theorem gives the regularity of a graph join.
	\begin{theorem}[\cite{Mousivand2012}]\label{join regularity formula}
		For graphs $G$ and $H$ on disjoint vertex sets,
		$$
		\reg(G\join H)=\max\{\reg(G),\reg(H)\}.
		$$
	\end{theorem}
	
	The following lemma is standard, but we record it because it drives the repeated-block analysis, and converts disjoint unions into tensor products of resolutions. 
	\begin{lemma}\label{lem:tensor}
		If $G$ and $H$ are graphs on disjoint vertex sets, then
		$$
		R_{G\sqcup H}/I(G\sqcup H)\cong (R_G/I(G))\otimes_K (R_H/I(H)).
		$$
		Consequently, if
		$
		\mathcal{B}_G(u,v):=\sum_{i,d\geq 0}\beta_{i,d}(G)u^iv^d
		 $ and $
		\mathcal{B}_H(u,v):=\sum_{i,d\geq 0}\beta_{i,d}(H)u^iv^d,
		$
		then
		$ \mathcal{B}_{G\sqcup H}(u,v)=\mathcal{B}_G(u,v)\mathcal{B}_H(u,v)$ and $ \reg(G\sqcup H)=\reg(G)+\reg(H). $
	\end{lemma}
	
	\begin{proof}
		Since the variable sets are disjoint, the quotient by $I(G\sqcup H)=I(G)+I(H)$ is exactly the tensor product of the two quotient rings over $K$. Let $F_\bullet$ and $L_\bullet$ be minimal graded free resolutions of $R_G/I(G)$ and $R_H/I(H)$ over $R_G$ and $R_H$, respectively. Then $F_\bullet\otimes_K L_\bullet$ is a minimal graded free resolution of $(R_G/I(G))\otimes_K(R_H/I(H))$ over the polynomial ring in the union of the two variable sets. The homological degrees and internal shifts add, and hence the Betti polynomial multiplies. The regularity statement follows, since $ \reg(M)=\max\{d-i:\beta_{i,d}(M)\neq 0\} $ and tensoring over disjoint variables adds the shifts.
	\end{proof}
	
	The next theorem packages the single-block Betti data proved in \cite{AnandGuptaRatherSingh2026}. It is the basic external input for our new results, which records the two-strand Betti table of the block $B_b$.
	\begin{theorem}[Theorems 3.1 and 3.11 \cite{AnandGuptaRatherSingh2026}]\label{thm:blockbetti}
		For $b\geq 3$, let
		$$
		\lambda_i=
		\begin{cases}
			2i\binom{b}{i+1}+b\binom{b-1}{i-1}, & i=1,\\[3mm]
			2i\binom{b}{i+1}+2b\binom{b-1}{i-1}+\binom{b}{2}, & i=3,\\[3mm]
			2i\binom{b}{i+1}+2b\binom{b-1}{i-1}, & \text{otherwise},
		\end{cases}
		$$
		where $1\leq i\leq 2b-2$, and let $\lambda_i=0$ outside this range. Let
		$$
		\mu_i=
		\lambda_{i+1}
		+\binom{2b-2}{i+2}
		-(2b-2)\binom{2b-2}{i+1}
		+(b^2-3b+1)\binom{2b-2}{i},
		$$
		with $1\leq i\leq 2b-2$. Then
		$ \beta_{i,j}(B_b)=0$ whenever $j-i\notin\{1,2\}, $
		and
		$ \beta_{i,i+1}(B_b)=\lambda_i,$  $ \beta_{i,i+2}(B_b)=\mu_i$, $ \reg(B_b)=2, $ and $ \pd(B_b)=2b-2. $
	\end{theorem}
	
	The next remark isolates the exceptional case $b=2$.
	\begin{remark}\label{rem:b2}
		 When $b=2$, the block $B_2$ is isomorphic to the cycle $C_4$. In that case $\reg(B_2)=1$, while the remaining structural arguments of this paper still apply after replacing the block data by the Betti table of $C_4$. To keep the formulas transparent, the explicit Betti-strand formulas in Sections~\ref{sec:linear} and \ref{sec:higher} are written under the standing assumption $b\geq 3$, and the case $b=2$ is mentioned separately where needed.
	\end{remark}
	
	The next proposition collects the join criteria that will be used later in the discussion of well-coveredness and shellability.
	
	\begin{proposition}[Proposition 2.8 \cite{AnandRoy2021}]\label{prop:joinproperties}
		 Let $G_1,\dots,G_t$ be graphs on pairwise disjoint vertex sets and let $G:=G_1\join \cdots \join G_t$. Then the following hold.
		\begin{enumerate}[label=\textup{(\roman*)},leftmargin=8mm]
			\item $G$ is well-covered if and only if each $G_i$ is well-covered and $\alpha(G_i)=\alpha(G_j)$ for all $i,j$.
			\item $G$ is vertex decomposable, shellable, and Cohen--Macaulay if and only if every $G_i$ is complete.
			\item $G$ is sequentially Cohen--Macaulay if and only if one factor is sequentially Cohen--Macaulay and all remaining factors are complete.
			\item If some factor has nonzero induced matching number, then
			$ \nu(G)=\max_i \nu(G_i). $
		\end{enumerate}
	\end{proposition}
	
	 The existing literature treats, in one direction, the single-block family $\overline{K}_a\join (K_b+K_2)$ and, in another direction, the repeated clique-block family $\overline{K}_a\join nK_b$ \cite{AnandGuptaRatherSingh2026}. What is missing is a parallel theory for the repeated mixed block $\overline{K}_a\join n(K_b+K_2)$. In particular, formulas for the independence polynomial, Hilbert series, repeated-block Betti convolution, and the resulting regularity and well-covered thresholds do not seem to have been written down. For $n=1$ in repeated mixed block, we recover results of \cite{AnandGuptaRatherSingh2026}. The present paper fills precisely that gap.
	
	\section{Independence complex, enumerators, and Hilbert series}\label{sec:independence}
	
	From now on, unless explicitly stated otherwise, we assume $a\geq 1$, $n\geq 1$, and $b\geq 3$. Let
	$$
	G_{a,b,n}:=MECS_{b,n}^a=\overline{K}_a\join (nB_b).
	$$
	Let $ A=\{u_1,\dots,u_a\} $ be the  the vertices of $\overline{K}_a$, and for each $1\leq s\leq n$ let
	$$
	V(B_b^{(s)})=\{x_1^{(s)},\dots,x_b^{(s)},y_1^{(s)},\dots,y_b^{(s)}\}
	$$
	be the vertices of the $s$-th copy of $B_b$. See $MECS_{3,2}^{2}$ in Figure \ref{fig:mecsgraph}.
	
	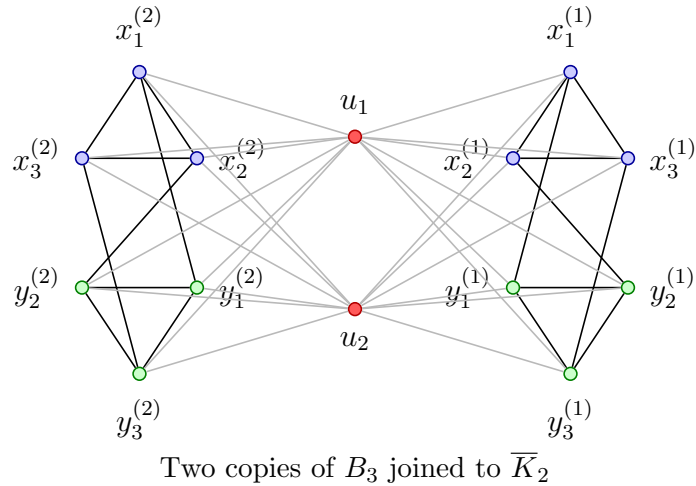
\begin{figure}[H]
		\centering
		\begin{tikzpicture}[scale=0.95,
			every node/.style={circle,inner sep=1.7pt,draw},
			redv/.style={fill=red!65,draw=red!70!black},
			bluev/.style={fill=blue!20,draw=blue!60!black},
			greenv/.style={fill=green!20,draw=green!50!black},
			line width=0.6pt]
			
			\node[redv,label=above:$u_1$] (u1) at (0,1.2) {};
			\node[redv,label=below:$u_2$] (u2) at (0,-1.2) {};
			
			\node[bluev,label=above:$x_1^{(1)}$] (x11) at (3,2.1) {};
			\node[bluev,label=left:$x_2^{(1)}$]  (x21) at (2.2,0.9) {};
			\node[bluev,label=right:$x_3^{(1)}$] (x31) at (3.8,0.9) {};
			
			\node[greenv,label=left:$y_1^{(1)}$]  (y11) at (2.2,-0.9) {};
			\node[greenv,label=right:$y_2^{(1)}$] (y21) at (3.8,-0.9) {};
			\node[greenv,label=below:$y_3^{(1)}$] (y31) at (3,-2.1) {};
			
			\node[bluev,label=above:$x_1^{(2)}$] (x12) at (-3,2.1) {};
			\node[bluev,label=right:$x_2^{(2)}$]  (x22) at (-2.2,0.9) {};
			\node[bluev,label=left:$x_3^{(2)}$] (x32) at (-3.8,0.9) {};
			
			\node[greenv,label=right:$y_1^{(2)}$]  (y12) at (-2.2,-0.9) {};
			\node[greenv,label=left:$y_2^{(2)}$] (y22) at (-3.8,-0.9) {};
			\node[greenv,label=below:$y_3^{(2)}$] (y32) at (-3,-2.1) {};
			
			\draw (x11)--(x21)--(x31)--(x11);
			\draw (y11)--(y21)--(y31)--(y11);
			\draw (x11)--(y11);
			\draw (x21)--(y21);
			\draw (x31)--(y31);
			
			\draw (x12)--(x22)--(x32)--(x12);
			\draw (y12)--(y22)--(y32)--(y12);
			\draw (x12)--(y12);
			\draw (x22)--(y22);
			\draw (x32)--(y32);
			
			\foreach \u in {u1,u2}{
				\foreach \v in {x11,x21,x31,y11,y21,y31,x12,x22,x32,y12,y22,y32}{
					\draw[gray!55] (\u)--(\v);
				}
			}
			
			\node[draw=none,rectangle,align=center] at (0,-3.4) {\small Two copies of $B_3$ joined to $\overline{K}_2$};
			
		\end{tikzpicture}
		\caption{A schematic picture of $MECS_{3,2}^{2}$.}
		\label{fig:mecsgraph}
	\end{figure}
	
	 The earlier split-like papers focus on homological quantities after the graph family has already been fixed \cite{AnandGuptaRatherSingh2026}. For the repeated extended block studied here, even the most basic structural parameters are not available in a single place. We discuss them briefly.
	
	Clearly, $ |V(G_{a,b,n})|=a+2bn $
		$ |E(G_{a,b,n})|=nb^2+2abn=nb(b+2a), $  and diameter  $
		\operatorname{diam}(G_{a,b,n})=2. $  The clique number of $G_{a,b,n}$ is $\omega(G_{a,b,n})=b+1$.  	For the chromatic number, color every vertex of $A$ with one new color, say color $0$. In each block $B_b$, color $x_i^{(s)}$ with color $i$ and color $y_i^{(s)}$ with color $i+1$ modulo $b$. Since the $x$-vertices form a clique, the $y$-vertices form a clique, and $x_i^{(s)}$ is adjacent to $y_i^{(s)}$, this is a proper $b$-coloring of each block. Distinct blocks may reuse the same $b$ colors because there are no edges between them. Thus $\chi(G_{a,b,n})\leq b+1$. Since $\chi(G)\geq \omega(G)$ for every graph, we conclude that $\chi(G_{a,b,n})=b+1$. 	For the independence number, no independent set can meet both $A$ and a block vertex, because the join connects every vertex of $A$ to every vertex outside $A$. Therefore every independent set lies either entirely in $A$ or entirely in $nB_b$. The set $A$ itself is independent of size $a$. In one block $B_b$, the largest independent sets have size $2$, namely sets $\{x_i^{(s)},y_j^{(s)}\}$ with $i\neq j$. Since the blocks are disjoint, $nB_b$ has independent sets of size $2n$, obtained by choosing such a pair in each block. Hence, we obtain  independence number as
		$ \alpha(G_{a,b,n})=\max\{a,2n\}. $
		Finally, with identity $\tau(G)+\alpha(G)=|V(G)|$, we have  the vertex cover number as
		$ \tau(G_{a,b,n})=a+2bn-\max\{a,2n\}. $

	The next proposition classifies minimal vertex covers of $G_{a,b,n}$.
	\begin{proposition}\label{prop:mincovers}
		 Let $W:=V(G_{a,b,n})\setminus A$. A subset $C\subseteq V(G_{a,b,n})$ is a minimal vertex cover if and only if one of the following holds: \textup{(i)} $C=W$; and \textup{(ii)} $A\subseteq C$ and $C\setminus A$ is a minimal vertex cover of $nB_b$. 
		Consequently, the minimal vertex cover sizes are $ 2bn$ and $ a+2n(b-1). $
	\end{proposition}
	
	\begin{proof}
		Suppose $C$ is a vertex cover. If some $u\in A$ and some $w\in W$ both lie outside $C$, then the edge ${u,w}$ is uncovered, which is impossible. Therefore every vertex cover contains all of $A$ or all of $W$. If $W\subseteq C$, then minimality forces $C=W$, because deleting any block vertex still leaves all vertices of $A$ uncovered with respect to that vertex. If $A\subseteq C$, then the only remaining edges to be covered are those inside $nB_b$, and therefore $C\setminus A$ must be a minimal vertex cover of $nB_b$.
		
		Conversely, both kinds of sets are easily seen to be vertex covers, and their minimality is immediate. It remains to compute the size in case \textup{(ii)}. In $G_{a,b,n}$, each block $B_b$ has independence number $2$, and hence the minimum vertex cover number is $2b-2$. Since the blocks are disjoint, a minimal cover of $nB_b$ is the disjoint union of minimal covers of the blocks and therefore has size $n(2b-2)$. Adding the $a$ vertices of $A$, we get the required second size formula.
	\end{proof}
	
	 When $n=1$, Proposition~\ref{prop:mincovers} reduces to the extended split-like case handled in \cite{AnandGuptaRatherSingh2026}. The new point here is that the repeated $B_b$-block structure preserves the dichotomy all block vertices versus all join vertices plus internal covers, but the threshold now depends on $2n$ rather than $2$.
	 The single-block paper \cite{AnandGuptaRatherSingh2026} works mainly from Hochster's formula and Hilbert-series comparison for one block. For the repeated family considered here, the independence complex admits a clean decomposition that does not seem to have been observed before. That decomposition immediately yields independent-set counts, purity information, the Hilbert series, and the top Betti number.
	
	The next proposition describes the independence complex of $G_{a,b,n}$.
	\begin{proposition}\label{prop:complexdecomp}
		 Let $\Gamma_b:=\Delta(B_b)$. Then
		$ \Gamma_b\cong C_{b,b} $ as a $1$-dimensional simplicial complex, and
		$ \Delta(G_{a,b,n})=\langle A\rangle \cup \Gamma_b^{\star n}, $
		where $\langle A\rangle$ is the simplex on the vertex set $A$ and $\Gamma_b^{\star n}$ is the $n$-fold simplicial join of $\Gamma_b$ with itself. The two components are disjoint on the vertex level. Furthermore, the following hold;
		\begin{enumerate}[label=\textup{(\roman*)},leftmargin=8mm]
			\item $\dim \Gamma_b=1$ and $\dim \Gamma_b^{\star n}=2n-1$;
			\item the facets of $\Gamma_b^{\star n}$ are precisely the unions of one edge of $\Gamma_b$ from each block;
			\item the number of facets of $\Gamma_b^{\star n}$ is $\big(b(b-1)\big)^n$;
			\item $\Delta(G_{a,b,n})$ is pure if and only if $a=2n$.
		\end{enumerate}
	\end{proposition}
	
	\begin{proof}
		By definition of $B_b$, an independent set in one block contains at most one $x$-vertex and at most one $y$-vertex, and the pair $\{x_i,y_i\}$ is forbidden, as it is a matching edge. Thus the $2$-element independent sets of $B_b$ are exactly the pairs $\{x_i,y_j\}$ with $i\neq j$. Hence $\Gamma_b$ is the crown graph $C_{b,b}$ as a $1$-dimensional simplicial complex. Furthermore, as $G_{a,b,n}$ is a join of $A$ with the disjoint union $nB_b$, every independent set lies entirely in $A$ or entirely in $nB_b$. This proves that
		$ \Delta(G_{a,b,n})=\langle A\rangle \cup \Delta(nB_b). $ Since independence complexes turn disjoint unions of graphs into simplicial joins, so we have
		$ \Delta(nB_b)=\Gamma_b^{\star n}. $
		The dimension formula follows from $\dim(\Delta_1\star \Delta_2)=\dim \Delta_1+\dim\Delta_2+1$.
		 Each facet of $\Gamma_b$ is an edge $\{x_i,y_j\}$ with $i\neq j$. A facet of $\Gamma_b^{\star n}$ is therefore the union of one such edge from each of the $n$ blocks, which shows that all these facets have size $2n$. Since $\Gamma_b$ has $b(b-1)$ edges, the total number of such facets equals $\big(b(b-1)\big)^n$.
		 The simplex $\langle A\rangle$ contributes one facet of size $a$, while every facet in $\Gamma_b^{\star n}$ has size $2n$. Therefore the whole complex is pure if and only if $a=2n$.
	\end{proof}

	Figure \ref{fig:ind-mecs-322} shows the independence simplicial complex of $MECS_{3,2}^{2}$ with its facet details.
	\begin{figure}[H]
		\centering
		\begin{tikzpicture}[
			scale=0.95,
			every node/.style={font=\small},
			vertex/.style={circle,inner sep=1.7pt,draw},
			redv/.style={vertex,fill=red!65,draw=red!70!black},
			bluev/.style={vertex,fill=blue!20,draw=blue!60!black},
			greenv/.style={vertex,fill=green!20,draw=green!50!black},
			ed/.style={line width=0.65pt},
			faint/.style={gray!45,dashed,line width=0.45pt},
			lab/.style={draw=none,rectangle,font=\small},
			biglab/.style={draw=none,rectangle,font=\Large},
			line width=0.6pt
			]
			
			
			\node[redv,label=above:$u_1$] (u1) at (0,0.8) {};
			\node[redv,label=below:$u_2$] (u2) at (0,-1.5) {};
			\draw[ed,red!70!black] (u1)--(u2);
			
			\node[lab] at (0,-2.6) {$\langle u_1,u_2\rangle$};
			
			
			\node[biglab] at (1.9,0) {$\cup$};
			
			
			\coordinate (C1) at (4.0,0);
			
			\node[bluev,label=above:$x_1^{(1)}$] (x11) at ($(C1)+(90:1.15)$) {};
			\node[greenv,label=above right:$y_2^{(1)}$] (y21) at ($(C1)+(30:1.15)$) {};
			\node[bluev,label=below right:$x_3^{(1)}$] (x31) at ($(C1)+(-30:1.15)$) {};
			\node[greenv,label=below:$y_1^{(1)}$] (y11) at ($(C1)+(-90:1.15)$) {};
			\node[bluev,label=below left:$x_2^{(1)}$] (x21) at ($(C1)+(-150:1.15)$) {};
			\node[greenv,label=above left:$y_3^{(1)}$] (y31) at ($(C1)+(150:1.15)$) {};
			
			\draw[ed] (x11)--(y21)--(x31)--(y11)--(x21)--(y31)--(x11);
			
			\node[lab] at (4.0,-2.6) {$\Gamma_3^{(1)}\cong C_6$};
			
			
			\node[biglab] at (6.35,0) {$*$};
			
			
			\coordinate (C2) at (8.7,0);
			
			\node[bluev,label=above:$x_1^{(2)}$] (x12) at ($(C2)+(90:1.15)$) {};
			\node[greenv,label=above right:$y_2^{(2)}$] (y22) at ($(C2)+(30:1.15)$) {};
			\node[bluev,label=below right:$x_3^{(2)}$] (x32) at ($(C2)+(-30:1.15)$) {};
			\node[greenv,label=below:$y_1^{(2)}$] (y12) at ($(C2)+(-90:1.15)$) {};
			\node[bluev,label=below left:$x_2^{(2)}$] (x22) at ($(C2)+(-150:1.15)$) {};
			\node[greenv,label=above left:$y_3^{(2)}$] (y32) at ($(C2)+(150:1.15)$) {};
			
			\draw[ed] (x12)--(y22)--(x32)--(y12)--(x22)--(y32)--(x12);
			
			\node[lab] at (8.7,-2.65) {$\Gamma_3^{(2)}\cong C_6$};
			
			
			\draw[faint,bend left=15] (x11) to (x12);
			\draw[faint,bend left=10] (y21) to (y22);
			\draw[faint,bend right=10] (x31) to (x32);
			\draw[faint,bend right=15] (y11) to (y12);
			
			\node[lab,align=center] at (7.35,2.5)
			{\footnotesize every edge of $\Gamma_3^{(1)}$ joined with every edge of $\Gamma_3^{(2)}$  gives a tetrahedral facet};
			
			
			\node[lab,align=center] at (4.6,-3.5)
			{$\displaystyle
				\Delta(MECS_{3,2}^{2})
				=
				\langle u_1,u_2\rangle
				\cup
				\left(\Gamma_3^{(1)}*\Gamma_3^{(2)}\right)
				$};
			
		\end{tikzpicture}
		\caption{Schematic representation of the independence simplicial complex of $MECS_{3,2}^{2}$.}
		\label{fig:ind-mecs-322}
	\end{figure}
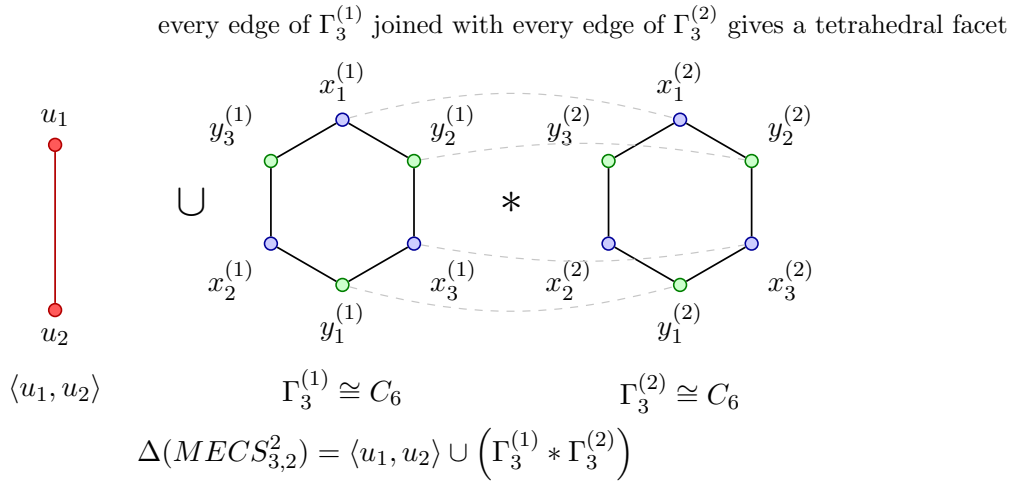

	The next theorem computes the independence polynomial.
	\begin{theorem}\label{thm:indpol}
		The independence polynomial of $G_{a,b,n}$ is
		$$
		\Ind(G_{a,b,n},z)=(1+z)^a+\big(1+2bz+b(b-1)z^2\big)^n-1.
		$$
	\end{theorem}
	
	\begin{proof}
		In $B_b$, there is exactly one independent set of size $0$, exactly $2b$ independent sets of size $1$, and exactly $b(b-1)$ independent sets of size $2$, namely the pairs ${x_i,y_j}$ with $i\neq j$. There are no larger independent sets, because each $x$-family and each $y$-family is a clique. Thus, we get $$
		\Ind(B_b,z)=1+2bz+b(b-1)z^2.
		$$
		 For disjoint unions, independent sets multiply blockwise, and hence the independence polynomial multiplies. So, from 
		$ \Ind(nB_b,z)=\Ind(B_b,z)^n,$ we obtain 
		$$
		\Ind(nB_b,z)=\big(1+2bz+b(b-1)z^2\big)^n.
		$$
		Finally, every independent set of $G_{a,b,n}$ lies either in $A$ or in $nB_b$, and the empty set is counted in both parts. Therefore, we obtain
		$$
		\Ind(G_{a,b,n},z)=\Ind(\overline{K}_a,z)+\Ind(nB_b,z)-1=(1+z)^a+\big(1+2bz+b(b-1)z^2\big)^n-1.
		$$
	\end{proof}
	
	The next corollary turns the polynomial identity into an explicit counting formula.
	\begin{corollary}\label{cor:sk}
		For $0\leq k\leq 2n$, the number of independent sets of size $k$ in $nB_b$ is
		$$
		s_k(nB_b)=
		\sum_{q=0}^{\lfloor k/2\rfloor}
		\binom{n}{q}
		\binom{n-q}{k-2q}
		(2b)^{k-2q}\big(b(b-1)\big)^q.
		$$
		For $k\geq 1$, the number of independent sets of size $k$ in $G_{a,b,n}$ is
		$ s_k(G_{a,b,n})=\binom{a}{k}+s_k(nB_b), $
		with $s_0(G_{a,b,n})=1$.
	\end{corollary}
	
	\begin{proof}
		To obtain an independent set of size $k$ in $nB_b$, choose $q$ blocks contributing a $2$-element independent set and $k-2q$ further blocks contributing a singleton. The first choice can be made in $\binom{n}{q}$ ways, the second in $\binom{n-q}{k-2q}$ ways, each singleton has $2b$ possibilities, and each $2$-set has $b(b-1)$ possibilities. Summing over $q$ gives the first formula.
		 The second statement follows immediately from Theorem~\ref{thm:indpol}, since positive-size independent sets of $G_{a,b,n}$ come either from the simplex on $A$ or from the repeated-block part.
	\end{proof}

	The next theorem gives the Hilbert series of the edge ring of $G_{a,b,n}$.
	\begin{theorem}\label{thm:hilbert}
		For $R:=R_{G_{a,b,n}}$, the Hilbert series is
		$$
		\Hilb_{R/I(G_{a,b,n})}(t)=
		\frac{1}{(1-t)^a}
		+
		\left(
		1+\frac{2bt}{1-t}+\frac{b(b-1)t^2}{(1-t)^2}
		\right)^n
		-1.
		$$
		Equivalently,
		$$
		\Hilb_{R/I(G_{a,b,n})}(t)=
		\frac{1}{(1-t)^a}
		+
		\frac{\big(1+(2b-2)t+(b-1)(b-2)t^2\big)^n}{(1-t)^{2n}}
		-1.
		$$
	\end{theorem}
	
	\begin{proof}
		For a simplicial complex $\Delta$, the Hilbert series of its Stanley--Reisner ring (see \cite[Chapter~5]{BrunsHerzog1998} or \cite[Chapter~II]{Stanley1996}) is given by
		$$
		\Hilb_{K[\Delta]}(t)=\sum_{F\in \Delta}\left(\frac{t}{1-t}\right)^{|F|}
		=\Ind(\Delta,\tfrac{t}{1-t}).
		$$
		 Since the edge ring of a graph is the Stanley--Reisner ring of its independence complex, Theorem~\ref{thm:indpol} yields
		\begin{align*}
		 \Hilb_{R/I(G_{a,b,n})}(t)&=
		\Ind\left(G_{a,b,n},\frac{t}{1-t}\right)\\
		&= \left(1+\frac{t}{1-t}\right)^a
		+
		\left(
		1+2b\frac{t}{1-t}
		+b(b-1)\frac{t^2}{(1-t)^2}
		\right)^n
		-1. 
		\end{align*}
	\end{proof}
	
	The next theorem determines the top Betti number, projective dimension, and depth.
	\begin{theorem}\label{thm:topbetti}
		 Let
		$ N=a+2bn$ be the order of  $G_{a,b,n})$.	Then
		$ \beta_{N-1,N}(G_{a,b,n})=1, \pd(G_{a,b,n})=N-1=a+2bn-1
		$
		and
		$
		\depth(R/I(G_{a,b,n}))=1.
		$
	\end{theorem}
	
	\begin{proof}
		By Proposition~\ref{prop:complexdecomp}, the independence complex $\Delta(G_{a,b,n})$ has exactly two connected components, namely the simplex on $A$ and the connected complex $\Gamma_b^{\star n}$. Therefore, we have
		$$
		\widetilde{H}_0(\Delta(G_{a,b,n});K)\cong K.
		$$
		Apply Hochster's formula to the full vertex set $V(G_{a,b,n})$, we obtain
		$$
		\beta_{N-1,N}(G_{a,b,n})
		= \dim_K \widetilde{H}_{N-(N-1)-1}(\Delta(G_{a,b,n});K)
		=
		 \dim_K \widetilde{H}_0(\Delta(G_{a,b,n});K)
		=
		1.
		$$
		Thus it follows that $\pd(G_{a,b,n})\geq N-1$. On the other hand, no quotient of a polynomial ring in $N$ variables can have projective dimension exceeding $N-1$, so $\pd(G_{a,b,n})=N-1$. The depth statement follows from Auslander--Buchsbaum identity
		$$
		\depth(R/I(G_{a,b,n}))=N-\pd(G_{a,b,n})=1.
		$$
	\end{proof}
	
	 For the single-block family, projective dimension was previously deduced through disconnectedness arguments \cite{AnandGuptaRatherSingh2026}. The proof above works directly from the independence-complex decomposition and therefore extends without additional machinery to the repeated-block setting.
	
	\medskip
	The next example illustrates the enumerative formulas.
	\begin{example}\label{ex:indpoly}
		For $G_{4,3,2}=MECS_{3,2}^{4}$, Theorem~\ref{thm:indpol} gives
		$$
		\Ind(G_{4,3,2},z)
		=(1+z)^4+(1+6z+6z^2)^2-1
		= 1+16z+54z^2+76z^3+37z^4.
		$$
		Hence there are $16$ independent vertices, $54$ independent pairs, $76$ independent triples, and $37$ independent $4$-sets. The Hilbert series is
		$$
		\Hilb_{R/I(G_{4,3,2})}(t)
		= \frac{1}{(1-t)^4}
		+
		\left(1+\frac{6t}{1-t}+\frac{6t^2}{(1-t)^2}\right)^2
		-1.
		$$
		By Theorem~\ref{thm:topbetti},
		$ \beta_{15,16}(G_{4,3,2})=1,
		\pd(G_{4,3,2})=15,$ and $
		\depth(R/I(G_{4,3,2}))=1. $
	\end{example}

	\section{Betti polynomials and the linear strand}\label{sec:linear}
	 The paper \cite{AnandGuptaRatherSingh2026} gives the Betti table of one block $B_b$, but it does not address repeated disjoint copies of $B_b$ joined to an independent set. The key new observation of this section is that the repeated-block family is governed by one compact Betti polynomial. Once that polynomial is identified, the linear strand of $MECS_{b,n}^a$ becomes explicit.
	
	The next definition introduces the Betti polynomial of a single block.
	\begin{definition}		
		For $b\geq 3$, define
		$$
		P_b(u,v):=
		1+\sum_{r=1}^{2b-2}\lambda_r u^r v^{r+1}
		+\sum_{r=1}^{2b-2}\mu_r u^r v^{r+2},
		$$
		where $\lambda_r$ and $\mu_r$ are given in Theorem~\ref{thm:blockbetti}.
	\end{definition}
	
	The next proposition identifies $P_b(u,v)$ with the Betti polynomial of one block.
	\begin{proposition}\label{prop:blockpoly}	
		For $b\geq 3$,
		$$
		P_b(u,v)=\sum_{i,d\geq 0}\beta_{i,d}(B_b)u^iv^d.
		$$
		In particular, the only nonzero terms of $P_b(u,v)$ have degree shift $d-i$ equal to $0$, $1$, or $2$.
	\end{proposition}
	
	\begin{proof}
		This is simply a restatement of Theorem~\ref{thm:blockbetti}. The constant term corresponds to $\beta_{0,0}(B_b)=1$, the terms of shift $1$ are $\beta_{i,i+1}(B_b)=\lambda_i$, and the terms of shift $2$ are $\beta_{i,i+2}(B_b)=\mu_i$.
	\end{proof}
	
	The next theorem gives the full Betti polynomial of the repeated block $nB_b$.
	\begin{theorem}\label{thm:powpoly}		
		For $b\geq 3$ and $n\geq 1$,
		$$
		\sum_{i,d\geq 0}\beta_{i,d}(nB_b)u^iv^d=P_b(u,v)^n,
		$$
		or equivalently, $ \beta_{i,d}(nB_b)=[u^iv^d],P_b(u,v)^n, $
		where $[u^iv^d]$ denotes coefficient extraction.
	\end{theorem}
	
	\begin{proof}
		By Lemma~\ref{lem:tensor}, disjoint union of graphs corresponds to multiplication of Betti polynomials. Since $nB_b$ is the disjoint union of $n$ copies of $B_b$, repeated application of the lemma implies
		$$
		\mathcal{B}_{nB_b}(u,v)=\mathcal{B}_{B_b}(u,v)^n=P_b(u,v)^n.
		$$
		The coefficient formula is just the meaning of multiplication in the polynomial ring $K[u,v]$.
	\end{proof}
	
	 Theorem~\ref{thm:powpoly} is the first compact generating formula for repeated $B_b$-blocks. In the clique-block family $\overline{K}_a\join nK_b$, only the linear strand is present in the repeated part \cite{AnandGuptaRatherSingh2026}. Here the repeated block already carries two nontrivial strands, so the polynomial $P_b(u,v)^n$ retains substantially more information.
	
	The next theorem transfers the repeated-block Betti data to the graph $G_{a,b,n}$.
	\begin{theorem}\label{thm:joinformula-new}
		For $b\geq 3$ and all $i,d\geq 0$,
		$$
		\beta_{i,d}(G_{a,b,n})
		= \sum_{j=0}^{d-2}\binom{a}{j}\beta_{i-j,d-j}(nB_b)
		+
		\delta_{d,i+1}\sum_{j=1}^{d-1}\binom{a}{j}\binom{2bn}{d-j},
		$$
		or equivalently,
		$$
		\beta_{i,d}(G_{a,b,n})
		= \sum_{j=0}^{d-2}\binom{a}{j}[u^{i-j}v^{d-j}]P_b(u,v)^n
		+
		\delta_{d,i+1}\sum_{j=1}^{d-1}\binom{a}{j}\binom{2bn}{d-j}.
		$$
	\end{theorem}
	
	\begin{proof}
		Apply the join formula of Theorem~\ref{Mousivand2012 theorem 2.2}  to
		$ G=\overline{K}_a,$ and $ H=nB_b. $
		The graph $\overline{K}_a$ has no edges, so its edge ideal is zero and its Betti table is trivial,
		$ \beta_{0,0}(\overline{K}_a)=1,$ and $ \beta_{i,d}(\overline{K}_a)=0$ otherwise.
		Thus only the $nB_b$-terms survive in the first sum of Theorem~\ref{Mousivand2012 theorem 2.2}, while the linear correction term remains unchanged. By Theorem~\ref{thm:powpoly}, the coefficient-extraction follows.
	\end{proof}
	
	The next corollary isolates the linear strand of the repeated block.
	\begin{corollary}\label{cor:linear-nBb}		
		For $b\geq 3$ and $i\geq 1$,
		$$
		\beta_{i,i+1}(nB_b)=n\lambda_i.
		$$
	\end{corollary}
	
	\begin{proof}
		In the expansion of $P_b(u,v)^n$, a term contributing to $\beta_{i,i+1}(nB_b)$ must have total shift $1$. Since every nonconstant term of $P_b(u,v)$ has shift either $1$ or $2$, the only way to achieve total shift $1$ is to choose exactly one linear-strand term from one factor and the constant term from all other factors. There are $n$ choices for the factor, and the coefficient contributed is $\lambda_i$. Hence, we obtain $\beta_{i,i+1}(nB_b)=n\lambda_i.$
	\end{proof}
	
	\medskip
	The next theorem gives the linear strand of $G_{a,b,n}$ explicitly.
	\begin{theorem}\label{thm:linear-G}
		For $b\geq 3$ and $i\geq 1$,
		$$
		\beta_{i,i+1}(G_{a,b,n})
		= n\sum_{j=0}^{i-1}\binom{a}{j}\lambda_{i-j}
		+
		\sum_{j=1}^{i}\binom{a}{j}\binom{2bn}{i-j+1}.
		$$
	\end{theorem}
	
	\begin{proof}
		For $d=i+1$ in Theorem~\ref{thm:joinformula-new}, we obtain
		$$
		\beta_{i,i+1}(G_{a,b,n})
		= \sum_{j=0}^{i-1}\binom{a}{j}\beta_{i-j,i-j+1}(nB_b)
		+
		\sum_{j=1}^{i}\binom{a}{j}\binom{2bn}{i-j+1}.
		$$
		Now , from Corollary~\ref{cor:linear-nBb}, the identity follows.
	\end{proof}
	
	The next corollary checks the linear formula against the edge count.
	\begin{corollary}\label{cor:edges}
		 For every $b\geq 3$,
		$$
		\beta_{1,2}(G_{a,b,n})=|E(G_{a,b,n})|=nb^2+2abn.
		$$
	\end{corollary}
	
	\begin{proof}
		By Theorem~\ref{thm:linear-G},
		$ \beta_{1,2}(G_{a,b,n})=n\lambda_1+a\binom{2bn}{1}. $
		From Theorem~\ref{thm:blockbetti},
		$ \lambda_1=2\binom{b}{2}+b=b^2. $
		Therefore, 	$ \beta_{1,2}(G_{a,b,n})=nb^2+2abn. $
	\end{proof}
	
	\medskip
	The next example shows the first linear-strand values for $G_{3,3,2}$.
	\begin{example}\label{ex:linear}	
		For $b=3$, Theorem~\ref{thm:blockbetti} gives $ \lambda_1=9, \lambda_2=16,$ and $ \lambda_3=9. $
		Hence, Theorem~\ref{thm:linear-G} gives
		 \begin{align*}
			\beta_{1,2}(G_{3,3,2})&=2\lambda_1+3\binom{12}{1}=18+36=54,\\
		\beta_{2,3}(G_{3,3,2})&=2\lambda_2+3\lambda_1+3\binom{12}{2}+3\binom{12}{1}=320,
		\end{align*}
		and $ \beta_{3,4}(G_{3,3,2})=1038. $
		These values agree with the sample computation listed later in Table~\ref{tab:betti-example}.
	\end{example}
	
	\section{Higher strands, regularity, and projective dimension}\label{sec:higher}
	 The clique-block family studied earlier has regularity growing linearly with the number of blocks \cite{AnandGuptaRatherSingh2026}. In contrast, each block $B_b$ already has regularity $2$, so repeated blocks create genuinely nonlinear strand propagation. This section makes that phenomenon explicit.

	The next theorem computes the quadratic strand of the repeated block.
	\begin{theorem}\label{thm:quadratic-nBb}		
		For $b\geq 3$ and $i\geq 1$,
		$$
		\beta_{i,i+2}(nB_b)
		= n\mu_i
		+
		\binom{n}{2}
		\sum_{\substack{p+q=i\\ p,q\geq 1}}
		\lambda_p\lambda_q.
		$$
	\end{theorem}
	
	\begin{proof}
		In the expansion of $P_b(u,v)^n$, a contribution to shift $2$ can arise in exactly two ways.
		 First, choose one shift-$2$ term from one factor and the constant terms from the remaining $n-1$ factors. This contributes
		$ n\mu_i. $ Second, choose two shift-$1$ terms from two distinct factors and the constant terms from the other $n-2$ factors. If the two homological degrees are $p$ and $q$ with $p+q=i$, the contribution is $\lambda_p\lambda_q$. There are $\binom{n}{2}$ ways to choose the two factors. Summing over all ordered decompositions $i=p+q$ gives the second term. No other choice produces total shift $2$, as every nonconstant term of $P_b$ has shift at least $1$.
	\end{proof}
	
	The next theorem transfers the quadratic strand to $G_{a,b,n}$.
	\begin{theorem}\label{thm:quadratic-G}
		 For $b\geq 3$ and $i\geq 1$,
		$$
		\beta_{i,i+2}(G_{a,b,n})
		= \sum_{j=0}^{i}\binom{a}{j}
		\left(
		n\mu_{i-j}
		+
		\binom{n}{2}\sum_{\substack{p+q=i-j\ p,q\geq 1}}\lambda_p\lambda_q
		\right).
		$$
	\end{theorem}
	
	\begin{proof}
		With $d=i+2$ in Theorem~\ref{thm:joinformula-new}, and note that $d\neq i+1$, the linear correction term vanishes and we obtain
		$$
		\beta_{i,i+2}(G_{a,b,n})
		= \sum_{j=0}^{i}\binom{a}{j}\beta_{i-j,i-j+2}(nB_b).
		$$
		Now, from Theorem~\ref{thm:quadratic-nBb}, the identity follows.
	\end{proof}
	
	The next proposition locates all possible nonzero strands.
	\begin{proposition}\label{prop:support}
		 For $b\geq 3$, if $\beta_{i,d}(G_{a,b,n})\neq 0$, then
		$ 1\leq d-i\leq 2n$ or equivalently,
		$$
		\beta_{i,d}(G_{a,b,n})=0
		\qquad\text{whenever}\qquad
		d<i+1 \ \text{or}\ d>i+2n.
		$$
	\end{proposition}
	
	\begin{proof}
		By Proposition~\ref{prop:blockpoly}, all nonconstant terms in $P_b(u,v)$ have shift $1$ or $2$. Therefore every nonconstant term in $P_b(u,v)^n$ has shift between $1$ and $2n$. In Theorem~\ref{thm:joinformula-new}, the first sum shifts both indices by the same amount $j$, so the difference $d-i$ is unchanged. The second term occurs only for $d=i+1$. Hence, there do not exists no new shift outside the interval $[1,2n].$
	\end{proof}
	
	The next theorem gives the regularity of the repeated family.
	\begin{theorem}\label{thm:reg}
		 If $b\geq 3$, then
		$ \reg(nB_b)=2n $ and $ \reg(G_{a,b,n})=2n. $
	\end{theorem}
	
	\begin{proof}
		By Theorem~\ref{thm:blockbetti}, $\reg(B_b)=2$. Then by Lemma~\ref{lem:tensor}, we obtain 
		$ \reg(nB_b)=n,\reg(B_b)=2n. $
		Now, from the join regularity formula of Theorem~\ref{join regularity formula}, we have
		$$
		\reg(G_{a,b,n})
		= \reg(\overline{K}_a\join nB_b)
		=
		\max\{\reg(\overline{K}_a),\reg(nB_b)\}.
		$$
		Since $\overline{K}_a$ has no edges, so its edge ideal is zero and therefore $\reg(\overline{K}_a)=0$. Thus, we obtain 
		$ \reg(G_{a,b,n})=2n. $
	\end{proof}
	
	The next remark records the exceptional $4$-cycle case.
	\begin{remark}
		 If $b=2$, then $B_2\cong C_4$ and $\reg(B_2)=1$. By the same tensor-product argument, we get
		$ \reg(nB_2)=n $  and  $\reg(MECS_{2,n}^a)=n. $
		Thus the block $K_b+K_2$ contributes regularity $2$ for $b\geq 3$, but only regularity $1$ in the degenerate case $b=2$.
	\end{remark}
	
	The next corollary shows that the one-block case is recovered correctly. In particular, Theorems~\ref{thm:linear-G}, \ref{thm:quadratic-G}, and \ref{thm:reg} recover the extended complete split-like family studied in \cite{AnandGuptaRatherSingh2026}. 
	\begin{corollary}\label{cor:ECSrecover}
		 When $n=1$,
		$$
		MECS_{b,1}^a=\overline{K}_a\join (K_b+K_2).
		$$
		In particular, for $b\geq 3$
		$$
		\reg(MECS_{b,1}^a)=2.
		$$
	\end{corollary}

	 For the clique-block family $\overline{K}_a\join nK_b$, the regularity is $n$ \cite{AnandGuptaRatherSingh2026}. Theorem~\ref{thm:reg} shows that replacing each clique block by the matched double-clique block $B_b$ doubles the regularity growth rate. This is one of the clearest algebraic signatures separating the two families.

	The next example computes the quadratic strand for $G_{3,3,2}$.
	\begin{example}\label{ex:quadratic}		
		For $b=3$, Theorem~\ref{thm:blockbetti} gives $ \mu_1=\mu_2=\mu_3=0,$ and $ \mu_4=1. $
		Hence, Theorem~\ref{thm:quadratic-nBb} gives
		$$
		\beta_{2,4}(2B_3)
		= \binom{2}{2}\lambda_1^2
		=
		81,
		$$
		since the only decomposition of $2$ into positive parts is $1+1$. Then Theorem~\ref{thm:quadratic-G} implies
		$ \beta_{2,4}(G_{3,3,2})=81,$ and $
		\beta_{3,5}(G_{3,3,2})=531. $
		Moreover, Theorem~\ref{thm:reg} gives $ \reg(G_{3,3,2})=4. $
	\end{example}
	
	\section{Matching, covering, and purity phenomena}\label{sec:matching}
	 The change from clique blocks to matched double-clique blocks also changes the covering thresholds. In the earlier family $\overline{K}_a\join nK_b$, the well-covered condition is $a=n$ \cite{AnandGuptaRatherSingh2026}. Here it becomes $a=2n$, reflecting the fact that each block $B_b$ contributes an independent set of size $2$.

	The next proposition determines the induced matching number of one block.
	\begin{proposition}\label{prop:nu-block}		
		For the block $B_b$,
		$$
		\nu(B_b)=
		\begin{cases}
			1, & b=2,3,\\
			2, & b\geq 4.
		\end{cases}
		$$
	\end{proposition}
	
	\begin{proof}
		If $b\geq 4$, the two edges ${x_1,x_2}$ and ${y_3,y_4}$ form an induced matching, as they are disjoint, and there is no cross edge between an endpoint of the first edge and an endpoint of the second. Hence $\nu(B_b)\geq 2$.
		 On the other hand, one cannot have an induced matching of size $3$ in $B_b$, since any three disjoint edges would necessarily include two edges meeting the same clique layer or one matching edge together with an edge in one of the two cliques, and in either case an extra adjacency appears among the six endpoints.
		 If $b=2$ or $3$, there are not enough vertices in the two clique layers to realize two disjoint clique edges with mutually disjoint index sets, so every two disjoint edges create an extra adjacency among their endpoints. Hence $\nu(B_b)=1$ in these cases.
	\end{proof}
	
	The next theorem computes the induced matching number of the whole family.
	\begin{theorem}\label{thm:nu-G}
		 For $G_{a,b,n}$,
		$$
		\nu(G_{a,b,n})=
		\begin{cases}
			n, & b=2,3,\\
			2n, & b\geq 4.
		\end{cases}
		$$
	\end{theorem}
	
	\begin{proof}
		Since $nB_b$ is the disjoint union of $n$ copies of $B_b$, induced matching numbers add across the components
		$ \nu(nB_b)=n,\nu(B_b). $
		Now apply Proposition~\ref{prop:joinproperties}\textup{(iv)} to the join
		$ G_{a,b,n}=\overline{K}_a\join nB_b, $ we obtain
		$$
		\nu(G_{a,b,n})=\nu(nB_b)=n,\nu(B_b),
		$$
		since $\nu(\overline{K}_a)=0$ and $\nu(nB_b)\neq 0$.
		Now, the result follows from Proposition~\ref{prop:nu-block}.
	\end{proof}
	
	The next theorem gives the exact well-covered and unmixed criterion.
	\begin{theorem}\label{thm:wellcovered}
		 For $G_{a,b,n}$, the following are equivalent:
		\begin{enumerate}[label=\textup{(\roman*)},leftmargin=8mm]
			\item $G_{a,b,n}$ is well-covered;
			\item $I(G_{a,b,n})$ is unmixed;
			\item $\Delta(G_{a,b,n})$ is pure;
			\item $a=2n$.
		\end{enumerate}
	\end{theorem}
	
	\begin{proof}
		By Proposition~\ref{prop:complexdecomp}, $\Delta(G_{a,b,n})$ is pure if and only if $a=2n$, so \textup{(iii)} and \textup{(iv)} are equivalent. The maximal independent sets of $G_{a,b,n}$ are exactly of two types, the set $A$ of size $a$, and the facets of $\Gamma_b^{\star n}$, all of which have size $2n$ by Proposition~\ref{prop:complexdecomp}. Therefore, all maximal independent sets have the same size if and only if $a=2n$. Hence \textup{(i)} and \textup{(iv)} are equivalent.
		 For graphs, unmixedness of the edge ideal is equivalent to all minimal vertex covers having the same size, which is in turn equivalent to all maximal independent sets having the same size. Alternatively, Proposition~\ref{prop:mincovers} shows that the only minimal cover sizes are $2bn$ and $a+2n(b-1)$, and these are equal exactly when $a=2n$. Thus \textup{(ii)} and \textup{(iv)} are equivalent as well.
	\end{proof}
	
	The next theorem shows that the family never becomes Cohen--Macaulay.
	\begin{theorem}\label{thm:notCM}		
		For every $a\geq 1$, $b\geq 2$, and $n\geq 1$, the graph $G_{a,b,n}$ is not Cohen--Macaulay.
	\end{theorem}
	
	\begin{proof}
		By Theorem~\ref{thm:topbetti}, we have
		$$
		\depth(R/I(G_{a,b,n}))=1.
		$$
		On the other hand, the Krull dimension of the edge ring equals the independence number of the graph, so we obtain
		$$
		\dim(R/I(G_{a,b,n}))=\alpha(G_{a,b,n})=\max\{a,2n\}\geq 2.
		$$
		Thus, we have
		$$
		\depth(R/I(G_{a,b,n}))<\dim(R/I(G_{a,b,n})),
		$$
		and therefore $R/I(G_{a,b,n})$ is not Cohen--Macaulay.
	\end{proof}
	
	The next theorem rules out several stronger topological properties.
	\begin{theorem}\label{thm:notshellable}
		 For every $a\geq 1$, $b\geq 2$, and $n\geq 1$, the graph $G_{a,b,n}$ is neither vertex decomposable nor shellable.
	\end{theorem}
	
	\begin{proof}
		With $ G_{a,b,n}=\overline{K}_a\join nB_b. $
		If $a\geq 2$, then $\overline{K}_a$ is not complete, and $nB_b$ is never complete, since it is disconnected with at least one nontrivial block. Therefore Proposition~\ref{prop:joinproperties}\textup{(ii)} implies that $G_{a,b,n}$ is neither vertex decomposable nor shellable. If $a=1$, then $\overline{K}_1$ is complete, but $nB_b$ is still not complete. Proposition~\ref{prop:joinproperties}\textup{(ii)} again implies the same conclusion.
	\end{proof}
	
	The next theorem gives a partial sequential Cohen--Macaulay obstruction.
	\begin{theorem}\label{thm:notSCM}
		 If $a\geq 2$, then $G_{a,b,n}$ is not sequentially Cohen--Macaulay.
	\end{theorem}
	
	\begin{proof}
		Since $a\geq 2$, the graph $\overline{K}_a$ is not complete. The graph $nB_b$ is also not complete for every $n\geq 1$ and $b\geq 2$. Therefore Proposition~\ref{prop:joinproperties}\textup{(iii)} shows that the join
		$ G_{a,b,n}=\overline{K}_a\join nB_b $
		cannot be sequentially Cohen--Macaulay.
	\end{proof}
	
	The next problem isolates the remaining sequential Cohen--Macaulay case.
	\begin{openproblem}
		 Determine whether $MECS_{b,n}^1$ is sequentially Cohen--Macaulay for some values of $b$ and $n$. Equivalently, determine the sequential Cohen--Macaulay behavior of the repeated block graph $nB_b$.
	\end{openproblem}
	
	 Theorem~\ref{thm:wellcovered} is the exact analogue of the well-covered criteria for the previously studied families, but with the threshold shifted from $a=n$ in the clique-block case to $a=2n$ here. This shift has a clean combinatorial explanation, as every $K_b$-block contributes one independent vertex, whereas every $B_b$-block contributes a maximal independent pair.
	
	The next example compares two parameter choices with different covering behavior.
	\begin{example}\label{ex:matching}
		 For $G_{4,4,2}=MECS_{4,2}^{4}$, Theorem~\ref{thm:wellcovered} implies 	$ a=4=2n, $ so the graph is well-covered and unmixed. Theorem~\ref{thm:nu-G} gives
		$ \nu(G_{4,4,2})=2n=4 $, since $b=4$.
		
		By contrast, for $G_{3,3,2}=MECS_{3,2}^{3}$, we  have $a=3\neq 4$, so the graph is neither pure nor well-covered, while
		$ \nu(G_{3,3,2})=n=2, $ since $b=3$.
	\end{example}
	
	\section{Algorithmic evaluation, examples, and comparison tables}\label{sec:algorithm}
	 The formulas of the previous sections are exact and closed, but one still wants a reproducible way to generate Betti data for concrete parameters. The present section turns the theory into a practical computation scheme and records sample outputs.
	
	\begin{algorithm}[H]
		\caption{Betti-data computation for $MECS_{b,n}^a$}
		\label{alg:mecs}
		\begin{algorithmic}[1]
			\Require Integers $a\geq 1$, $b\geq 2$, $n\geq 1$
			\Ensure Structural invariants and nonzero graded Betti numbers of $MECS_{b,n}^a$
			\State Construct the block data of $B_b$
			\If{$b\geq 3$}
			\State Compute $\lambda_i$ and $\mu_i$ from Theorem~\ref{thm:blockbetti}
			\State Form $P_b(u,v)=1+\sum \lambda_i u^iv^{i+1}+\sum \mu_i u^iv^{i+2}$
			\Else
			\State Replace $P_b(u,v)$ by the Betti polynomial of $C_4$
			\EndIf
			\State Compute $P_b(u,v)^n$ by repeated sparse convolution
			\State Extract $\beta_{i,d}(nB_b)$ from the coefficients of $P_b(u,v)^n$
			\State Apply Theorem~\ref{thm:joinformula-new} to obtain $\beta_{i,d}(MECS_{b,n}^a)$
			\State Compute
			$$
			|V|,\ |E|,\ \alpha,\ \omega,\ \chi,\ \tau,\ \nu,\ \reg,\ \pd
			$$
			from Section \ref{sec:independence}, Proposition \ref{prop:mincovers}, Theorems \ref{thm:nu-G}, \ref{thm:reg}, and \ref{thm:topbetti}
			\State Output the Betti table, the independence polynomial, and the comparison summary
		\end{algorithmic}
	\end{algorithm}
	
	The next proposition gives a coarse complexity estimate for Algorithm~\ref{alg:mecs}.
	\begin{proposition}\label{prop:complexity}
		 Fix $b$ and let $M:=2b-2$. Using naive sparse convolution on the support of $P_b(u,v)$, the first two Betti strands of $MECS_{b,n}^a$ can be computed in $O(n^2M^2)$ scalar operations, while a full expansion of $P_b(u,v)^n$ requires at most $O(n^3M^3)$ scalar operations. The subsequent join transformation contributes only lower-order binomial convolutions.
	\end{proposition}
	
	\begin{proof}
		The polynomial $P_b(u,v)$ has $O(M)$ nonconstant terms. To obtain the linear and quadratic strands, only combinations with total shift at most $2$ need to be tracked, so each convolution step uses $O(M^2)$ pairings, and there are $O(n^2)$ such pairings across the repeated multiplications. For the full table, the support broadens with the number of blocks, and a naive bound is cubic in the maximal homological range, hence $O(n^3M^3)$. The join formula in Theorem~\ref{thm:joinformula-new} is then just a finite sum of coefficient lookups weighted by binomial coefficients.
	\end{proof}
	
	\begin{figure}[H]
		\centering
		\begin{tikzpicture}[
			node distance=8mm and 10mm,
			box/.style={rectangle,rounded corners,draw=blue!60!black,fill=blue!7,minimum width=36mm,minimum height=9mm,align=center},
			arrow/.style={-{Latex[length=2.3mm]},thick}
			]
			\node[box] (input) {Input $(a,b,n)$};
			\node[box,below=of input] (block) {Compute block data\ $\lambda_i,\mu_i$ or $C_4$ data};
			\node[box,below=of block] (poly) {Build $P_b(u,v)$};
			\node[box,below=of poly] (power) {Repeated sparse convolution\ to obtain $P_b(u,v)^n$};
			\node[box,below=of power] (join) {Apply join transform\ Theorem~\ref{thm:joinformula-new}};
			\node[box,below=of join] (output) {Output Betti table,\ $\Ind(G,z)$, $\reg(G)$, $\pd(G)$, $\nu(G)$, etc.};
			
			\draw[arrow] (input) -- (block);
			\draw[arrow] (block) -- (poly);
			\draw[arrow] (poly) -- (power);
			\draw[arrow] (power) -- (join);
			\draw[arrow] (join) -- (output);
		\end{tikzpicture}
		\caption{A flow diagram for the computation of the homological data of $MECS_{b,n}^a$.}
		\label{fig:flow}
	\end{figure}
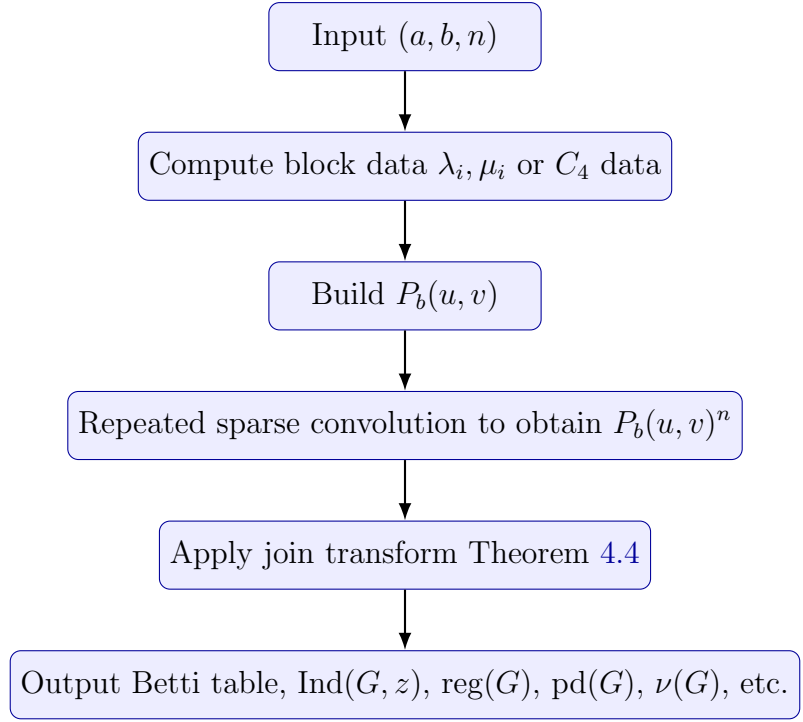
	
	Figure \ref{fig:flow} shows the computation flow of the homological data of $MECS_{b,n}^a$ (see Algorithm \ref{alg:mecs}). The scheme separates the repeated-block step from the join step, which is conceptually useful and computationally efficient.
	
	\medskip
	The next example gives a concrete Betti computation.
	\begin{example}\label{ex:mainexample}
		 For $ G_{3,3,2}=MECS_{3,2}^{3},$  Examples~\ref{ex:linear} and \ref{ex:quadratic}, together with Theorems~\ref{thm:topbetti}, \ref{thm:reg}, and \ref{thm:nu-G}, we obtain
		$ |V(G_{3,3,2})|=15,  |E(G_{3,3,2})|=54, 
		\alpha(G_{3,3,2})=4, 
		\omega(G_{3,3,2})=\chi(G_{3,3,2})=4,
		\reg(G_{3,3,2})=4, 
		\pd(G_{3,3,2})=14, 
		\nu(G_{3,3,2})=2, $
		and the graph is not well-covered because $a=3\neq 2n=4$.
		 The first nonzero Betti numbers are
		$
		\beta_{1,2}=54, 
		\beta_{2,3}=320, 
		\beta_{2,4}=81, 
		\beta_{3,4}=1038, 
		\beta_{3,5}=531,
		$
		while the top entry is
		$
		\beta_{14,15}=1.
		$
		A longer list appears in Table~\ref{tab:betti-example}.
	\end{example}
	\begin{table}[H]
		\centering
		\caption{Selected nonzero graded Betti numbers of $G_{3,3,2}=MECS_{3,2}^{3}$.}
		\label{tab:betti-example}
		\begin{tabular}{ccc}
			\toprule
			Homological degree $i$ & Nonzero $d$ & $\beta_{i,d}(G_{3,3,2})$ \\
			\midrule
			$1$ & $2$ & $54$ \\
			$2$ & $3,4$ & $320,\ 81$ \\
			$3$ & $4,5$ & $1038,\ 531$ \\
			$4$ & $5,6$ & $2379,\ 1527$ \\
			$5$ & $6,7,8$ & $4167,\ 2493,\ 18$ \\
			$6$ & $7,8,9$ & $5661,\ 2493,\ 86$ \\
			$7$ & $8,9,10$ & $5940,\ 1527,\ 168$ \\
			$8$ & $9,10,11,12$ & $4785,\ 531,\ 168,\ 1$ \\
			$14$ & $15$ & $1$ \\
			\bottomrule
		\end{tabular}
	\end{table}
	
	The next proposition summarizes how the present family sits relative to earlier split-like families.
	\begin{proposition}\label{prop:specialization}
		 The graph family $MECS_{b,n}^a$ satisfies the following specialization pattern.
		\begin{enumerate}[label=\textup{(\roman*)},leftmargin=8mm]
			\item If $n=1$, then $MECS_{b,1}^a=\overline{K}_a\join (K_b+K_2)$ is the extended complete split-like graph.
			\item If $K_b+K_2$ is replaced by $K_b$, then one obtains the multiple complete split-like graph $MCS_{b,n}^a$ from \cite{AnandGuptaRatherSingh2026}.
			\item In passing from $MCS_{b,n}^a$ to $MECS_{b,n}^a$, the independence contribution of each block changes from $1$ to $2$, and the regularity contribution of each block changes from $1$ to $2$ for $b\geq 3$.
		\end{enumerate}
	\end{proposition}
	
	\begin{proof}
		Statements \textup{(i)} and \textup{(ii)} are immediate from the definitions. Statement \textup{(iii)} follows from the fact that $\alpha(K_b)=1$ whereas $\alpha(B_b)=2$, and from the regularity formulas $\reg(K_b)=1$ and $\reg(B_b)=2$ for $b\geq 3$, together with additivity under disjoint union.
	\end{proof}
	
	Table \ref{tab:comparison} gives the invariants comparison for several split-like graph families. The last row is the new family $MECS_{b,n}^a$ treated in our study.
	\begin{table}[H]
		\centering
		\small 
		\setlength{\tabcolsep}{4pt} 
		\caption{Comparison of several split-like graph families.}
		\label{tab:comparison}
		\begin{tabularx}{\textwidth}{@{} l c c c c >{\centering\arraybackslash}X l @{}}
			\toprule
			Family & $|V|$ & $\alpha$ & $\operatorname{reg}$ & $\operatorname{pd}$ & Well-covered criterion & Comment \\
			\midrule
			$CS_b^a$ & $a+b$ & $\max\{a,1\}=a$ & $1$ & $a+b-1$ & $a=1$ & complete split \\
			$ECS_b^a$ & $a+2b$ & $\max\{a,2\}$ & $2$ & $a+2b-1$ & $a=2$ & one $B_b$ block \\
			$MCS_{b,n}^a$ & $a+bn$ & $\max\{a,n\}$ & $n$ & $a+bn-1$ & $a=n$ & repeated clique blocks \\
			$MECS_{b,n}^a$ & $a+2bn$ & $\max\{a,2n\}$ & $2n$ & $a+2bn-1$ & $a=2n$ & repeated $B_b$ blocks \\
			\bottomrule
		\end{tabularx}
	\end{table}
	
\section{Properties of independence polynomial of $G_{a,b,n}.$}\label{ind section}
		
		Throughout this section, we write
		$$
		P_{a,b,n}(z):=\Ind(G_{a,b,n},z)
		=(1+z)^a+\bigl(1+2bz+b(b-1)z^2\bigr)^n-1.
		$$
		 For brevity put $q_b(z):=1+2bz+b(b-1)z^2. $ Thus $ P_{a,b,n}(z)=(1+z)^a+q_b(z)^n-1. $
		 The subtraction of $1$ is important, as it avoids counting the empty independent set twice, once from the simplex $\langle A\rangle$ and once from $\Gamma_b^{\star n}$.
		
		The following result gives the coefficient formula for $P_{a,b,n}(z).$
		\begin{theorem}
			Let $ P_{a,b,n}(z)=\sum_{k\geq 0} c_k z^k. $ Then $c_0=1$, and for $k\geq 1$,
			 $$
			c_k=\binom{a}{k}
			+
			\sum_{j=\max\{0,k-n\}}^{\lfloor k/2\rfloor}
			\frac{n!}{j!(k-2j)!(n-k+j)!}
			(2b)^{k-2j}\bigl(b(b-1)\bigr)^j,
			$$
			where $\binom{a}{k}=0$ for $k>a$, or equivalently,
			 $ c_k=\binom{a}{k}+[z^k]q_b(z)^n $, for $k\geq 1$.
		\end{theorem}
		
		\begin{proof}
			We have $ q_b(z)^n=\bigl(1+2bz+b(b-1)z^2\bigr)^n. $ To obtain a term of degree $k$, choose $j$ of the $n$ factors to contribute $b(b-1)z^2$, choose $k-2j$ of the remaining factors to contribute $2bz$, and let the remaining $n-k+j$ factors contribute $1$. Therefore the contribution is
			 $$
			\frac{n!}{j!(k-2j)!(n-k+j)!}
			(2b)^{k-2j}\bigl(b(b-1)\bigr)^j.
			$$
			 The integer $j$ must satisfy $j\geq 0$, $k-2j\geq 0$, and $n-k+j\geq 0$. Hence
			 $
			\max\{0,k-n\}\leq j\leq \lfloor k/2\rfloor. $ The term $(1+z)^a$ contributes $\binom{a}{k}$ to the coefficient of $z^k$. Finally, the constant term is $ 1+1-1=1. $ This proves the formula.
		\end{proof}
		 
		 For the graph $MECS_{3,2}^{2}$ we have $a=2$, $b=3$, and $n=2$. Hence
		 $$
		P_{2,3,2}(z)
		=
		1+14z+49z^2+72z^3+36z^4.
		$$
		The above polynomial is unimodal and log-concave.
			
		 We observe that, although each block polynomial $q_b(z)^n$ is real-rooted, the full independence polynomial
		 $$
		P_{a,b,n}(z)=(1+z)^a+q_b(z)^n-1
		$$
		 need not be real-rooted, need not be log-concave, and need not be unimodal for all admissible choices of $a,b,n$.
		 Consider  the following polynomial $$
			P_{4,4,1}(z)
			=
			(1+z)^4+\bigl(1+8z+12z^2\bigr)-1
			=1+12z+18z^2+4z^3+z^4.
			$$
			  For log-concavity we would need $
			4^2\geq 18\cdot 1,$ but $
			16<18. $ Therefore,  the sequence of $P_{4,4,1}(z) $ is not log-concave. 
		For $a=8$, $b=6$, and $n=1$, we have 
			$$
			P_{8,6,1}(z)
			=
			(1+z)^8+\bigl(1+12z+30z^2\bigr)-1
			 =
			1+20z+58z^2+56z^3+70z^4+56z^5+28z^6+8z^7+z^8.
			$$
			 The coefficient sequence is  $1,\ 20,\ 58,\ 56,\ 70,\ 56,\ 28,\ 8,\ 1,$ and it  rises from $1$ to $58$, then drops to $56$, then rises again to $70$. Hence it is not unimodal. 
			 
			 \begin{theorem}\label{thm:unimodal-logconcave-Pabn}
			 	Let
			 	$ P_{a,b,n}(z):=(1+z)^a+\bigl(1+2bz+b(b-1)z^2\bigr)^n-1
			 	=\sum_{k=0}^{D} p_k z^k, $
			 	where $D=\max\{a,2n\}$ and, with the conventions $\binom{a}{k}=0$ for $k<0$ or $k>a$ and $q_k=0$ for $k<0$ or $k>2n$,
			 	$ p_0=1,  p_k=\binom{a}{k}+q_k$  for $k\ge 1, $  	with
			 	$$
			 	q_k=\sum_{j=0}^{\lfloor k/2\rfloor}
			 	\binom{n}{j}\binom{n-j}{k-2j}(2b)^{k-2j}\bigl(b(b-1)\bigr)^j.
			 	$$
			 	
			 	Then the following hold.
			 	
			 	\begin{enumerate}[label=\textup{(\roman*)},leftmargin=8mm]
			 		\item $P_{a,b,n}(z)$ is \emph{log-concave} if and only if
			 		$ p_k^2\ge p_{k-1}p_{k+1}$  for every $1\le k\le s, $ where
			 		$ s=\min\{D-1,\min{a,2n}+1\}. $
			 		Equivalently,
			 		$$
			 		\left(\binom{a}{k}+q_k\right)^2
			 		\ge
			 		\left(\binom{a}{k-1}+q_{k-1}\right)\left(\binom{a}{k+1}+q_{k+1}\right)
			 		\quad (1\le k\le s).
			 		$$
			 		\item $P_{a,b,n}(z)$ is \emph{unimodal} if and only if there exists an index $t\in\{0,1,\dots,D\}$ such that
			 		$ p_0\le p_1\le \cdots \le p_t\ge p_{t+1}\ge \cdots \ge p_D. $
			 		Equivalently, if $\nabla_k:=p_k-p_{k-1}$ for $1\le k\le D$, then $P_{a,b,n}(z)$ is unimodal if and only if the sequence $\{\nabla_k\}_{k=1}^{D}$ has at most one sign change, and that sign change is from $+$ to $-$.
			 	\end{enumerate}
			 	 In particular, whenever the inequalities in part $\textup{(i)}$ hold, the polynomial $P_{a,b,n}(z)$ is automatically unimodal.
			 \end{theorem}
			 
			 \begin{proof}
			 	In each factor of $ \bigl(1+2bz+b(b-1)z^2\bigr)^n $, we choose either $1$, or $2bz$, or $b(b-1)z^2$. If exactly $j$ quadratic terms and exactly $k-2j$ linear terms are selected, then the contribution to the coefficient of $z^k$ is
			 	$$
			 	\binom{n}{j}\binom{n-j}{k-2j}(2b)^{k-2j}\bigl(b(b-1)\bigr)^j.
			 	$$
			 	Summing over all admissible $j$ gives the displayed formula for $q_k$, we obtain
			 	$$
			 	P_{a,b,n}(z)=1+\sum_{k\ge 1}\left(\binom{a}{k}+q_k\right)z^k,
			 	$$
			 	which is exactly the coefficient description stated above.
			 	
			 	\smallskip
			 	We note that $1+2bz+b(b-1)z^2=(1+(b+\sqrt b)z)(1+(b-\sqrt b)z). $ 	Thus all zeros of $1+2bz+b(b-1)z^2$ are real and negative, and so all zeros of its $n$th power are again real and negative. By Newton's inequalities \cite{stanleyuniumodal}, the coefficient sequence $(q_k)_{k=0}^{2n}$ is log-concave, and hence unimodal. The binomial sequence $\Biggl(\binom{a}{k}\Biggr)_{k=0}^{a}$ is also log-concave and unimodal.
			 	
			 	\smallskip
			 	Now, we prove part $\textup{(i)}$. The necessity is immediate from the definition of log-concavity. For sufficiency, it is enough to show that outside the range $1\le k\le s$ the inequalities
			 	$ p_k^2\ge p_{k-1}p_{k+1} $ 	hold automatically. There are three cases. (1), If $a=2n$, then $D=a=2n$ and $s=D-1$. So the displayed inequalities are checked for all admissible $k$, and there is nothing more to prove. (2), If $a<2n$, then $D=2n$ and $s=a+1$. For every $k\ge a+2$, we have $\binom{a}{k-1}=\binom{a}{k}=\binom{a}{k+1}=0$, so
			 	$ p_{k-1}=q_{k-1}, p_k=q_k, $ and $p_{k+1}=q_{k+1}. $
			 	Since $(q_k)$ is log-concave, we obtain
			 	$ p_k^2=q_k^2\ge q_{k-1}q_{k+1}=p_{k-1}p_{k+1}$ and $k\ge a+2.$ Hence only the indices $1\le k\le a+1=s$ need to be checked.
			 	
			 	\medskip
			 	 \noindent (3) If $a>2n$, then $D=a$ and $s=2n+1$. For every $k\ge 2n+2$, we have $q_{k-1}=q_k=q_{k+1}=0$, so
			 	$ p_{k-1}=\binom{a}{k-1}, p_k=\binom{a}{k}, p_{k+1}=\binom{a}{k+1}. $
			 	Since the binomial sequence is log-concave,
			 	$$
			 	p_k^2=\binom{a}{k}^2\ge \binom{a}{k-1}\binom{a}{k+1}=p_{k-1}p_{k+1}
			 	\qquad (k\ge 2n+2).
			 	$$
			 	Hence only the indices $1\le k\le 2n+1=s$ need to be checked. Combining the above three cases, part $\textup{(i)}$ follows.
			 	
			 	\medskip
			 	 We turn to part $\textup{(ii)}$. This is simply the coefficient characterization of unimodality, a finite coefficient sequence $(p_0,p_1,\dots,p_D)$ is unimodal if and only if it increases weakly up to some index $t$ and then decreases weakly afterwards. This is equivalent to saying that the first-difference sequence
			 	$ \nabla_k=p_k-p_{k-1} $
			 	has at most one sign change, and if such a change occurs, it is from positive to negative. This gives the stated necessary and sufficient condition. Finally, the last sentence follows from the elementary fact that every nonnegative log-concave sequence without internal zeros is unimodal \cite{stanleyuniumodal}. Here all coefficients are positive, so there are no internal zeros.
			 \end{proof}
			 
			 \begin{remark}\label{rem:practical-check}
			 	For actual computations, Theorem~\ref{thm:unimodal-logconcave-Pabn} is useful, as the log-concavity test is finite and short. One does  not need to inspect all indices up to $D-1$, only the overlap window
			 	$ 1\le k\le \min\{D-1,\min\{a,2n\}+1\} $ matters. Beyond that point one of the two summands disappears, and the remaining tail is automatically log-concave.
			 \end{remark}
			 
			 For $(a,b,n)=(4,3,2)$, the polynomial
			 	$$
			 	P_{4,3,2}(z)=(1+z)^4+(1+6z+6z^2)^2-1
			 	=1+16z+54z^2+76z^3+37z^4.
			 	$$
			 	 is log-concave, and hence unimodal.
			 	 
			 For  $(a,b,n)=(6,3,2)$, the polynomial is
			 	$$
			 	P_{6,3,2}(z)=(1+z)^6+(1+6z+6z^2)^2-1
			 	=1+18z+63z^2+92z^3+51z^4+6z^5+z^6.
			 	$$
			 	The above polynomial is clearly unimodal, but not log-concave,  as 
			 	$$
			 	p_5^2=6^2=36\ngeq p_4p_6=51\cdot 1=51.
			 	$$
			 
			 \begin{example}
			 	For $(a,b,n)=(42,3,11)$, a direct expansion gives the coefficient sequence
			 	\begin{align*}
			 		&(1,108,2907,51080,648510,6260820,47365138,286323480,1401497865,5606278690,\\
			 		&18445035741,50130718944,112899673496,211292658960,330089247960,434783744608,\\
			 		&494029802754,507077978868,503905179930,513288104400,534410573436,542248642056,\\
			 		&\dots).
			 	\end{align*}
			 	Around the middle, we have
			 	$$
			 	507077978868>503905179930<513288104400<534410573436<542248642056,
			 	$$
			 	so after an initial drop the coefficients rise again. Hence the sequence is \emph{not} unimodal, and by Theorem~\ref{thm:unimodal-logconcave-Pabn}, it is therefore not log-concave either.
			 \end{example}
			 
			 The detailed analysis of above values is summarized in Table \ref{tab 1}.
			 \begin{table}[H]
			 	\centering
			 	\caption{Illustration of the three possible behaviors of $P_{a,b,n}(z)$.}
			 	\label{tab 1}
			 	\begin{tabular}{c c c c}
			 		\hline
			 		$(a,b,n)$ & coefficient sequence of $P_{a,b,n}(z)$ & unimodal? & log-concave? \\
			 		\hline
			 		$(4,3,2)$ & $(1,16,54,76,37)$ & Yes & Yes \\
			 		$(6,3,2)$ & $(1,18,63,92,51,6,1)$ & Yes & No \\
			 		$(42,3,11)$ & non-unimodal near the middle & No & No \\
			 		\hline
			 	\end{tabular}
			 \end{table}
			 
			 \begin{corollary}\label{cor:easy-sufficient}
			 	A convenient sufficient condition for unimodality is that, 	if
			 	$$
			 	\left(\binom{a}{k}+q_k\right)^2
			 	\ge
			 	\left(\binom{a}{k-1}+q_{k-1}\right)\left(\binom{a}{k+1}+q_{k+1}\right)
			 	\quad\text{for all } 1\le k\le s,
			 	$$
			 	where $s=\min\{D-1,\min\{a,2n\}+1\}$, then $P_{a,b,n}(z)$ is unimodal.
			 \end{corollary}
			 
			 \begin{proof}
			 	This is immediate from Theorem~\ref{thm:unimodal-logconcave-Pabn}(i), as log-concavity implies unimodality for positive coefficient sequences.
			 \end{proof}
			 
			 \begin{remark}
			 	The above theorem shows that there is no universal parameter condition of the form $a\le f(b,n)$ or $a\ge g(b,n)$ that governs unimodality or log-concavity in all cases. The behavior is genuinely mixed: some triples $(a,b,n)$ give log-concavity, some give unimodality without log-concavity, and some fail even unimodality. The exact answer is therefore coefficient-theoretic, and Theorem~\ref{thm:unimodal-logconcave-Pabn} gives that exact criterion.
			 \end{remark}
		
		The following theorem gives the elementary localization of zeros.		
		\begin{theorem}\label{enstrom}
			Let $ P_{a,b,n}(z)=\sum_{k=0}^d c_k z^k, $ where $d=\max\{a,2n\}$. Since every coefficient $c_k$ is positive, $P_{a,b,n}(z)$ has no positive real zero. Moreover, every zero $\zeta$ of $P_{a,b,n}(z)$ satisfies the Eneström--Kakeya bound
			
			$$
			\min_{0\leq k\leq d-1}\frac{c_k}{c_{k+1}}
			\leq
			|\zeta|
			\leq
			\max_{0\leq k\leq d-1}\frac{c_k}{c_{k+1}}.
			$$
			
		\end{theorem}
		
		\begin{proof}
			Since all coefficients $c_k$ are positive, we have $
			P_{a,b,n}(x)>0, $ for every real $x>0$. Hence there is no positive real zero. The stated annular bound is the classical Eneström--Kakeya theorem \cite{marden} applied to the polynomial $
			P_{a,b,n}(z)=c_0+c_1z+\cdots+c_dz^d $ with $c_k>0$ for every $k$.
		\end{proof}
		
		For $P_{2,3,2}(z)=1+14z+49z^2+72z^3+36z^4$,   every zero $\zeta$ satisfies
		$ \frac{1}{14}\leq |\zeta|\leq 2. $ Numerically, the moduli of the zeros are approximately
		 $ 1.0000, 0.1035, 0.5180, 0.5180, $ which agrees with the annular bound. Table \ref{tab:ind-polynomial-numerics} gives the numerical data for certain values of the polynomial $P_{a,b,n}(z)$.
		\begin{table}[H]
			\centering
			\footnotesize 
			\setlength{\tabcolsep}{3pt} 
			\caption{Numerical data for selected independence polynomials $P_{a,b,n}(z)$.}
			\label{tab:ind-polynomial-numerics}
			\begin{tabularx}{\textwidth}{@{} c c c >{\centering\arraybackslash}X c c >{\centering\arraybackslash}X @{}}
				\toprule
				$a$ & $b$ & $n$ & $P_{a,b,n}(z)$ & unimodal & log-concave & zeros, rounded \\
				\midrule
				$2$ & $3$ & $2$ & $1+14z+49z^2+72z^3+36z^4$ & yes & yes & $-1, -0.1035, -0.4482\pm0.2596i$ \\ \addlinespace
				$1$ & $3$ & $1$ & $1+7z+6z^2$ & yes & yes & $-1, -0.1667$ \\ \addlinespace
				$2$ & $3$ & $1$ & $1+8z+7z^2$ & yes & yes & $-1, -0.1429$ \\ \addlinespace
				$3$ & $3$ & $1$ & $1+9z+9z^2+z^3$ & yes & yes & $-7.8730, -1, -0.1270$ \\ \addlinespace
				$4$ & $4$ & $1$ & $1+12z+18z^2+4z^3+z^4$ & yes & no & $-0.0972, -0.6651, -1.6189\pm3.5842i$ \\ \addlinespace
				$8$ & $6$ & $1$ & $1+20z+58z^2+56z^3+70z^4+56z^5+28z^6+8z^7+z^8$ & no & no & $-0.0598, -0.3986, 0.2031\pm0.9977i, -1.2349\pm1.8277i, -2.7390\pm0.9019i$ \\ \addlinespace
				$2$ & $4$ & $2$ & $1+18z+89z^2+192z^3+144z^4$ & yes & yes & $-0.6591, -0.0854, -0.2944\pm0.1916i$ \\ \addlinespace
				$5$ & $3$ & $3$ & $1+23z+136z^2+442z^3+761z^4+649z^5+216z^6$ & yes & yes & $-1, -0.0622, -0.7690\pm0.2687i, -0.2022\pm0.2670i$ \\
				\bottomrule
			\end{tabularx}
		\end{table}

		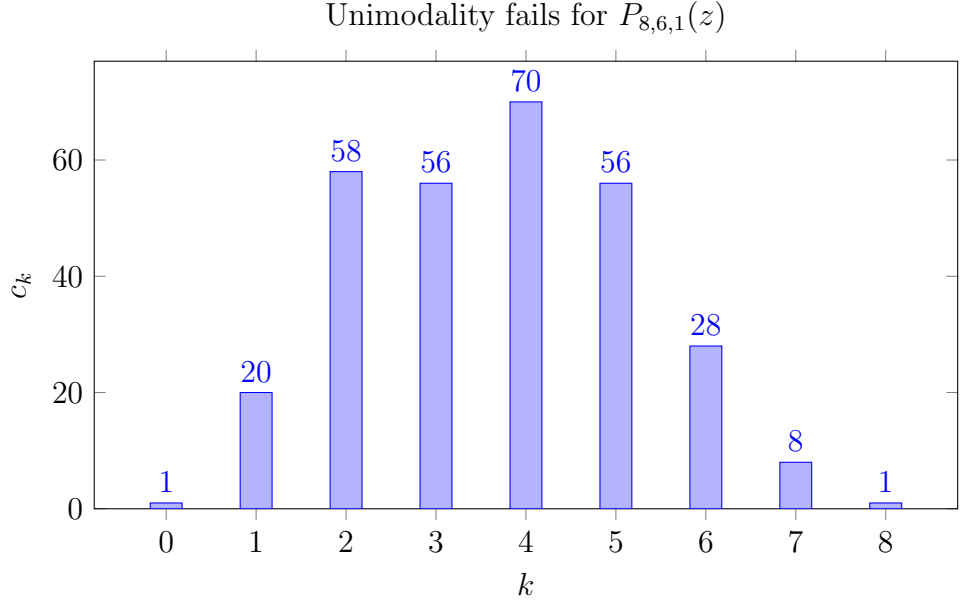
\begin{figure}[H]
			\centering
			\begin{tikzpicture}
				\begin{axis}[
					ybar,
					width=13cm,
					height=7.5cm,
					xlabel={$k$},
					ylabel={$c_k$},
					symbolic x coords={0,1,2,3,4,5,6,7,8},
					xtick=data,
					ymin=0,
					bar width=12pt,
					nodes near coords,
					nodes near coords align={vertical},
					title={Unimodality fails for $P_{8,6,1}(z)$}
					]
					\addplot coordinates {
						(0,1)
						(1,20)
						(2,58)
						(3,56)
						(4,70)
						(5,56)
						(6,28)
						(7,8)
						(8,1)
					};
				\end{axis}
			\end{tikzpicture}
			\caption{The sequence rises, falls, and rises again: $58>56<70$.}
			\label{fig:unimodality-fails}
		\end{figure}
		Figure \ref{fig:unimodality-fails} shows the unimodal failure of $P_{8,6,1}(z)$, its polynomial and zeros are shown in Table \ref{tab:ind-polynomial-numerics}.	 
		
		 Since all coefficients are positive, so by Enestr\"om--Kakeya theorem version Theorem \ref{enstrom}, we have 
			$ \min_{0\le k\le 7}\tfrac{a_k}{a_{k+1}}=\tfrac1{20},$ and $
			\max_{0\le k\le 7}\tfrac{a_k}{a_{k+1}}=8. $
			Thus, by Theorem \ref{enstrom}, every zero $\zeta$ of $P_{8,6,1}(z)$ lies in the annulus
			$ \tfrac1{20}\le |\zeta|\le 8. $ Figure \ref{fig:roots-annulus} shows its zero in plane, where the shaded region indicates the annulus $2\le |z|\le 8$.
		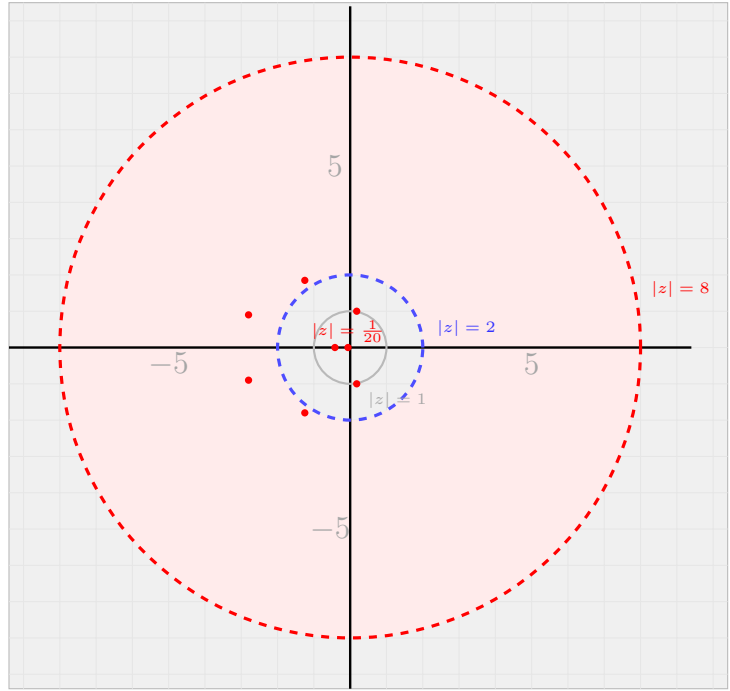
\begin{figure}[H]
			\centering
			\begin{tikzpicture}[scale=0.48, >=Latex]
				
				\def\Rout{8}
				\def\Rmid{2}
				\def\Rin{1}
				
				\fill[gray!12] (-9.4,-9.4) rectangle (10.4,9.5);
				\draw[gray!55] (-9.4,-9.4) rectangle (10.4,9.5);
				
				\draw[step=1.0, very thin, gray!20] (-9.4,-9.4) grid (10.4,9.5);
				
				\fill[red!8, even odd rule] (0,0) circle (\Rout) (0,0) circle (\Rmid);
				
				\draw[black, line width=0.9pt] (-9.4,0) -- (9.4,0);
				\draw[black, line width=0.9pt] (0,-9.4) -- (0,9.4);
				
				\draw[red, dashed, line width=1.2pt] (0,0) circle (\Rout);
				\draw[blue!70, dashed, line width=1.2pt] (0,0) circle (\Rmid);
				\draw[gray!55, line width=0.8pt] (0,0) circle (\Rin);
				
				\node[gray!70] at (-5,-0.45) {$-5$};
				\node[gray!70] at (5,-0.45) {$5$};
				\node[gray!70] at (-0.42,5) {$5$};
				\node[gray!70] at (-0.55,-5) {$-5$};
				
				
				\node[red, font=\bfseries] at (9.1,1.6) {\tiny $|z|=8$};
				\node[blue!80, font=\bfseries] at (3.2,0.55) {\tiny$|z|=2$};
				\node[gray!70, font=\bfseries] at (1.3,-1.45) {\tiny $|z|=1$};
				\node[red, font=\bfseries] at (-0.05,0.45) {\tiny $|z|=\tfrac{1}{20}$};
				
				\fill[red] (-2.80,  0.90) circle (0.10);
				\fill[red] (-2.80, -0.90) circle (0.10);
				\fill[red] (-1.25,  1.85) circle (0.10);
				\fill[red] (-1.25, -1.80) circle (0.10);
				\fill[red] ( 0.18,  1.00) circle (0.10);
				\fill[red] ( 0.18, -1.00) circle (0.10);
				\fill[red] (-0.42,  0.00) circle (0.10);
				\fill[red] (-0.06,  0.00) circle (0.10);
				
			\end{tikzpicture}
			\caption{A schematic plot of the zeros of the independence polynomial in the complex plane, together with the circles $|z|=1$, $|z|=2$, and $|z|=8$.}
			\label{fig:roots-annulus}
		\end{figure}
	\section{Conclusion and future work}\label{sec:conclusion}
	 We studied the family
	$ MECS_{b,n}^a\cong \overline{K}_a\join \big(n(K_b+K_2)\big), $
	which extends the previously known one-block extended complete split-like graphs and, at the same time, provides a repeated-block counterpart to the clique-based family $\overline{K}_a\join nK_b$ \cite{AnandGuptaRatherSingh2026}. The central observation is that the block $B_b=K_b+K_2$ already has a two-strand Betti table, and this feature survives under repeated disjoint union through tensor products of minimal free resolutions. That mechanism leads naturally to the Betti polynomial $P_b(u,v)^n$ for $nB_b$, and after the join transformation one obtains explicit formulas for the linear and quadratic strands of $MECS_{b,n}^a$.
	
	We described the independence complex as the disjoint union of a simplex and an $n$-fold simplicial join of crown complexes, from which the independence polynomial and the Hilbert series follow immediately. This decomposition also made it possible to identify the top Betti number directly and therefore to compute the projective dimension and depth without appealing to ad hoc disconnectedness arguments. In addition, we determined the induced matching number, classified the well-covered and unmixed members by the simple threshold $a=2n$, and proved that the family is never Cohen--Macaulay.
	
	A useful conceptual outcome of the paper is the comparison with previously known split-like families. Replacing a clique block by the matched double-clique block $B_b$ changes two basic growth rates at once, the independence contribution of one block increases from $1$ to $2$, and for $b\geq 3$ the regularity contribution increases from $1$ to $2$. This shift is visible both combinatorially and homologically and explains why the threshold for well-coveredness changes from $a=n$ in the clique-block setting to $a=2n$ here. Also, detailed analytical analysis of independence polynomial is investigated.
	
	The present article still has limitations. Although the repeated-block Betti polynomial gives a complete implicit description of all graded Betti numbers, only the linear and quadratic strands were written out explicitly. Likewise, the sequential Cohen--Macaulay behavior of the extremal case $a=1$ was not resolved in full generality. These two issues suggest several directions for future work:
	\begin{enumerate}[label=\textup{(\roman*)},leftmargin=8mm]
		\item obtain closed formulas for all higher Betti strands of $MECS_{b,n}^a$;
		\item determine the precise sequential Cohen--Macaulay criterion for $nB_b$ and hence for $MECS_{b,n}^1$;
		\item investigate symbolic powers, Rees algebras, and Castelnuovo functions for this family;
		\item extend the construction to hypergraph analogues where each block is replaced by a higher-dimensional split-like module.
	\end{enumerate}
	These problems appear to be both accessible and structurally meaningful, and they should further clarify how repeated graph modules control the homological shape of edge ideals.
	
	\section*{Declarations}
	\noindent \textbf{Data Availability:} There is no data associated with this article.
	
	\noindent \textbf{Funding:} The authors did not receive support from any organization for the submitted work.
	
	\noindent \textbf{Conflict of interest:} The authors have no competing interests to declare that are relevant to the content of this article.


\begin{thebibliography}{99}
		
		\bibitem{AnandGuptaRatherSingh2026}
		S. Anand, N. Gupta, S. A. Rather and P. Singh, Homological invariants of some complete split-like graphs, \emph{Beitr. Algebra Geom.} (2025), \url{https://doi.org/10.1007/s13366-025-00819-5}.
		
		\bibitem{AnandRoy2021}
		S. Anand and A. Roy, Graded Betti numbers of some families of circulant graphs, \emph{Rocky Mountain J. Math.} \textbf{51}(6) (2021) 1919--1940.
		
		\bibitem{AnandSinghVats2025}
		S. Anand, P. Singh and R. Vats, Graded Betti numbers of some split hypergraphs, \emph{Rocky Mountain J. Math.} \textbf{55}(4) (2025) 899--910.
		
		\bibitem{BrunsHerzog1998}
		W. Bruns and J. Herzog, \emph{Cohen--Macaulay Rings}, 2nd ed., Cambridge Stud. Adv. Math. 39, Cambridge Univ. Press, Cambridge, 1998.
		
		\bibitem{DaoSchweig2013}
		H. Dao and J. Schweig, Projective dimension, graph domination parameters, and independence complex homology, \emph{J. Combin. Theory Ser. A} \textbf{120}(2) (2013) 453--469.
		
		\bibitem{FranciscoHaVanTuyl2009}
		C. A. Francisco, H. T. Hà and A. Van Tuyl, Splittings of monomial ideals, \emph{Proc. Amer. Math. Soc.} \textbf{137}(10) (2009) 3271--3282.
		
		\bibitem{Froberg1990}
		R. Fröberg, On Stanley--Reisner rings, in \emph{Topics in Algebra, Part 2}, Banach Center Publ. 26, PWN, Warsaw, 1990, pp. 57--70.
		
		\bibitem{GuptaSinghAnand2026}
		N. Gupta, P. Singh and S. Anand, Certain homological invariants of some bipartite graphs, \emph{J. Algebra Appl.} \textbf{25}(10) (2026), Art. No. 2650111.
		
		\bibitem{HaVanTuyl2008}
		H. T. Hà and A. Van Tuyl, Monomial ideals, edge ideals of hypergraphs, and their graded Betti numbers, \emph{J. Algebraic Combin.} \textbf{27}(2) (2008) 215--245.
		
		\bibitem{HaWoodroofe2014}
		H. T. Hà and R. Woodroofe, Results on the regularity of square-free monomial ideals, \emph{Adv. Appl. Math.} \textbf{58} (2014) 21--36.
		
		\bibitem{HerzogHibi2011}
		J. Herzog and T. Hibi, \emph{Monomial Ideals}, Grad. Texts in Math. 260, Springer, London, 2011.
		
		\bibitem{HerzogHibiZheng2004}
		J. Herzog, T. Hibi and X. Zheng, Dirac's theorem on chordal graphs and Alexander duality, \emph{European J. Combin.} \textbf{25}(7) (2004) 949--960.
		
		\bibitem{Hochster1975}
		M. Hochster, Cohen--Macaulay rings, combinatorics, and simplicial complexes, in \emph{Ring Theory II}, Lecture Notes in Pure and Appl. Math. 26, Dekker, New York, 1977, pp. 171--223.
		
		\bibitem{Jacques2004}
		S. Jacques, \emph{Betti Numbers of Graph Ideals}, Ph.D. thesis, University of Sheffield, 2004.
		
		\bibitem{Jonsson2008}
		J. Jonsson, \emph{Simplicial Complexes of Graphs}, Lecture Notes in Math. 1928, Springer, Berlin, 2008.
		
		\bibitem{Katzman2006}
		M. Katzman, Characteristic-independence of Betti numbers of graph ideals, \emph{J. Combin. Theory Ser. A} \textbf{113}(3) (2006) 435--454.
		
		
		\bibitem{marden} . Marden, \textit{Geometry of Polynomials}, American Mathematical Society, Mathematical Surveys and Monographs, No. 3, (1966).
		\bibitem{MillerSturmfels2005}
		E. Miller and B. Sturmfels, \emph{Combinatorial Commutative Algebra}, Grad. Texts in Math. 227, Springer, New York, 2005.
		
		\bibitem{MoreyVillarreal2012}
		S. Morey and R. H. Villarreal, Edge ideals: algebraic and combinatorial properties, in \emph{Progress in Commutative Algebra 1}, de Gruyter, Berlin, 2012, pp. 85--126.
		
		\bibitem{Mousivand2012}
		A. Mousivand, Algebraic properties of product of graphs,
		\emph{Comm. Algebra} \textbf{40} (2012) 4177--4194. 
		
		\bibitem{bilafilo} B. A. Rather, Betti numbers of edge ideals of some graphs with application to graphs assigned to groups, \textit{Filomat},  \textbf{38}(6) (2024) 2185--2204.
		\bibitem{bilaljco} B. A. Rather, Homological invariants of edge ideals of Wollastonite graphs, \textit{J. Combin. Optimiz.} \textbf{51} (2026), Art. N. 36.
		\bibitem{bilalfilo1} B. A. Rather, and J. Wang, Homological invariants of edge ideals of power graphs of  finite groups, \textit{Filomat} \textbf{40}(2) (2026) 739--750.
		
		\bibitem{Stanley1996}
		R. P. Stanley, \emph{Combinatorics and Commutative Algebra}, 2nd ed., Progr. Math. 41, Birkhäuser, Boston, MA, 1996.
		
		\bibitem{stanleyuniumodal} R. P. Stanley, Log-concave and unimodal sequences in algebra, combinatorics, and geometry, \textit{Ann. New York Acad. Sci} \textbf{576}(1) (1989) 500--535.
		
		\bibitem{Villarreal2015}
		R. H. Villarreal, \emph{Monomial Algebras}, 2nd ed., Monogr. Res. Notes Math., CRC Press, Boca Raton, FL, 2015.
		
		\bibitem{Woodroofe2009}
		R. Woodroofe, Vertex decomposable graphs and obstructions to shellability, \emph{Proc. Amer. Math. Soc.} \textbf{137}(10) (2009) 3235--3246.
		
	\end{thebibliography}
\end{document}